\documentclass[11pt]{amsart}
\usepackage{amssymb,amscd,  latexsym, graphicx, mathrsfs, enumerate}
\usepackage{color}

\newcommand{\nc}{\newcommand}
\numberwithin{equation}{section}
\newtheorem{theorem}{Theorem}[section]
\newtheorem{prop}[theorem]{Proposition}
\newtheorem{importnota}[theorem]{Important Notation}
\newtheorem{prblm}[theorem]{Problem}
\newtheorem{notation}[theorem]{Notation}
\newtheorem{caution}[theorem]{Caution}
\newtheorem{remark}[theorem]{Remark}
\newtheorem{lemma}[theorem]{Lemma}
\newtheorem{construction}[theorem]{Construction}
\newtheorem{corollary}[theorem]{Corollary}
\newtheorem{example}[theorem]{Example}
\newtheorem{conclusion}[theorem]{Conclusion}
\newtheorem{triviality}[theorem]{Triviality}
\newtheorem{proto}[theorem]{Prototype Quasifibration}
\newtheorem{cauex}[theorem]{Cautionary Example}
\newtheorem{propositiondef}[theorem]{Proposition-Definition}
\newtheorem{subth}{Nuisance}[theorem]
\newtheorem{ssubth}{ }[subth]
\newtheorem{conjecture}[theorem]{Conjecture}
\newtheorem{sidest}[theorem]{Side Story}
\newtheorem{miniexample}[theorem]{Example}
\theoremstyle{definition}
\newtheorem{defin}[theorem]{Definition}

\nc\tri[1]{\begin{triviality}}
\nc\side[1]{\begin{sidest}}
\nc\conj[1]{\begin{conjecture}}
\nc\prodef[1]{\begin{propositiondef}}
\nc\prt[1]{\begin{proto}}
\nc\lem[1]{\begin{lemma}}
\nc\sblm[1]{\begin{sublemma}}
\nc\pro[1]{\begin{prop}}
\nc\thm[1]{\begin{theorem}}
\nc\cor[1]{\begin{corollary}}
\nc\dfn[1]{\begin{defin}}
\nc\sthm[1]{\begin{subth}}
\nc\exm[1]{\begin{example}}
\nc\miniexm[1]{\begin{miniexample}}
\nc\plm[1]{\begin{prblm}}
\nc\rmk[1]{\begin{remark}}
\nc\subrmk[1]{\begin{subremark}}
\nc\ntn[1]{\begin{notation}}
\nc\cau[1]{\begin{caution}}
\nc\imn[1]{\begin{importnota}}
\nc\cax[1]{\begin{cauex}}
\nc\con[1]{\begin{construction}}
\nc\ssthm[1]{\begin{ssubth}}
\nc\cnc[1]{\begin{conclusion}}
\nc\elem{\end{lemma}}
\nc\esblm{\end{sublemma}}
\nc\eside{\end{sidest}}
\nc\econj{\end{conjecture}}
\nc\eprodef{\end{propositiondef}}
\nc\eprt{\end{proto}}
\nc\ethm{\end{theorem}}
\nc\ecor{\end{corollary}}
\nc\edfn{\end{defin}}
\nc\esthm{\end{subth}}
\nc\epro{\end{prop}}
\nc\etri{\end{triviality}}
\nc\eexm{\end{example}}
\nc\eminiexm{\end{miniexample}}
\nc\ermk{\end{remark}}
\nc\subermk{\end{subremark}}
\nc\eplm{\end{prblm}}
\nc\ecau{\end{caution}}
\nc\ecax{\end{cauex}}
\nc\eimn{\end{importnota}}
\nc\entn{\end{notation}}
\nc\econ{\end{construction}}
\nc\ecnc{\end{conclusion}}
\nc\essthm{\end{ssubth}}

\newcommand{\C}{\mathbb{C}}
\newcommand{\R}{\mathbb{R}}
\newcommand{\Q}{\mathbb{Q}}

\newcommand{\Z}{\mathbb{Z}}

\newcommand{\A}{\mathbb{A}}

\newcommand{\ve}{\varepsilon}
\newcommand{\diag}{{\rm diag}}

\newcommand{\ds}{\displaystyle}

\newcommand{\f}{\bold{f}}

\newcommand{\lra}{\longrightarrow}

\newcommand{\GL}{{\rm GL}}
\newcommand{\SL}{{\rm SL}}
\newcommand{\SO}{{\rm SO}}
\newcommand{\SU}{{\rm SU}}
\newcommand{\bs}{\backslash}

\renewcommand{\Bbb}{\mathbb}

\title[Eisenstein series on $\SO(3,n+1)$]
{Fourier coefficients of Eisenstein series on $\SO(3,n+1)$}
\author{Henry H. Kim and Takuya Yamauchi}
\date{\today}
\thanks{The first author is partially supported by NSERC grant \#482564. 
}
\address{Henry H. Kim \\
Department of mathematics \\
 University of Toronto \\
Toronto, Ontario M5S 2E4, CANADA \\
and Korea Institute for Advanced Study, Seoul, KOREA}
\email{henrykim@math.toronto.edu}

\address{Takuya Yamauchi \\
Mathematical Inst. Tohoku Univ.\\
 6-3,Aoba, Aramaki, Aoba-Ku, Sendai 980-8578, JAPAN}
\email{takuya.yamauchi.c3@tohoku.ac.jp}


\begin{document}
\maketitle
\begin{abstract}
We explicitly compute Fourier coefficients of Eisenstein series on the special orthogonal group 
$G=\SO(3,n+1)$ over $\Q$ with $n\ge 2$ which 
splits everywhere at finite places. We show that it has a bounded denominator.
\end{abstract}

\tableofcontents



\section{Introduction}\label{intro}

Explicit computations of Fourier coefficients of Eisenstein series on reductive groups have been an important theme in the theory of automorphic forms. 
Among many works in this direction, the explicit formula for the Fourier coefficients of Siegel Eisenstein series obtained by several people (see, for example, \cite{Katsurada} among others) has played a crucial role in the explicit construction of cuspidal automorphic forms on symplectic groups, known as Ikeda lifts (\cite{Ik01}). 

Along with further developments in the explicit theory of Fourier coefficients of Eisenstein series, and by making use of previously established formulas, similar constructions have been extended to other reductive groups (\cite{Ik08}, \cite{Y10}, \cite{KY}). 
In this paper, we investigate the case where the reductive group is $\SO(3,n+1)$ and 
the results here will be used in the forthcoming paper \cite{KY3}. 

Let $\A=\A_\Q$ be the ring of adeles of $\Q$. 
For an integer $n\ge 2$, let $A$ be a positive definite integral matrix of size $n-2$. 
Assume $A$ is even (namely, any diagonal entries are even integers) if $n$ is even while $A=\diag(1,A')$ if $n$ is odd where $A'$ is 
a positive definite integral matrix of size $n-3$.  
Let $J_{2,n}={\rm antidiag}(1,1,-A,1,1)$ and 
$J_{3,n+1}={\rm antidiag}(1,1,1,-A,1,1,1)$ which are of size $n+3$ and $n+4$ 
respectively (see Section \ref{pre}). 
Let $G=\SO(3,n+1)$ be the special orthogonal group over $\Q$ associated to 
the symmetric pairing defined by $J_{3,n+1}$. Note that if $n$ is even, $G$ does not have discrete series.
Assume $G$ splits everywhere at finite places.  
By \cite[Section 2.1]{Serre}, it is equivalent to 
$\det(A)=1$ and $n\equiv 2$ mod 8 when $n$ is even, 
$\det(A')=1$ and $n\equiv 3$ mod 8 when $n$ is odd. 
For instance, this condition is satisfied when $n=8a+2$,  
$A=\diag(\overbrace{A_8,\ldots,A_8}^{a})$ or $n=8a+3$, 
$A=\diag(1,\overbrace{A_8,\ldots,A_8}^{a})$
where $A_8$ is the $E_8$ Cartan matrix given as an element of even integral matrix in $M_8(\Z)$. This agrees with the result of \cite{G} since $\frac {n+4}2\equiv 3$ (mod 4) if $n$ is even, and $\frac {n+3}2\equiv 3$ (mod 4) if $n$ is odd. 

Let $P$ be the Siegel parabolic subgroup of $G$ with the Levi decomposition 
$P=MN$ where $M=\{\diag(t,m,t^{-1})\ |\ t\in \GL_1,\ m\in \SO(2,n)\}\simeq \GL_1\times \SO(2,n)$ and $N\simeq \mathbb{G}^{n+2}_a$ where $\SO(2,n)$ is defined by $J_{2,n}$. 
Let $V'$ be the quadratic space associated to $J_{2,n}$ with the quadratic map 
$q:V'\lra \mathbb{G}_a$ and we identify $V'$ with $N$. 
We also denote by $(\ast,\ast)_{V'}$ the corresponding symmetric bilinear pairing. 
Let $V''$ be the quadratic space associated to $J_{1,n-1}=
{\rm antidiag}(1,-A,1)$.

In \cite[Section 4]{Po}, Pollack defined the Eisenstein series $E_l$ on $G(\A)$ of weight $l$ with respect to $G(\widehat{\Z})$ and computed the unramified and archimedean parts of its Fourier coefficients. He also proved the algebraicity of these Fourier coefficients. 

Building on his work, we give a more explicit computation of the Siegel series arising from the bad finite places. Since this computation is rather involved, we postpone it to Section \ref{Siegel}. As a consequence, we show that the Fourier coefficients of the Eisenstein series have bounded denominators. More precisely,
for $g=\gamma k_f g_\infty\in G(\A)=G(\Q)(G(\widehat{\Z})\times G(\R))$, 
$$
E_l(g)=E_0(g)+\sum_{0\neq\eta\in V'(\Q)\atop q(\eta)> 0}a_{E_l}(\eta)\mathcal W_{2\pi \eta}(g_\infty),
$$
where $E_0(g)$ is the constant term, and $\mathcal W_{2\pi \eta}$ is Pollack's spherical function 
associated to the additive character $e^{2\pi \sqrt{-1}(\eta,\ast)_{V'}}$ on $V'(\R)$ 
(see \cite[Section 3]{Po} but we use a slightly different normalization as explained later). We remark that $a_{E_l}(\eta)=0$ unless 
$\eta\in V'(\Z)$. 

\begin{theorem}\label{mainthms}{\rm(}Theorem \ref{FCEisenEven} and Theorem \ref{FCEisenOdd}{\rm)}  Let $l>n+1$ be an even integer. 
\begin{enumerate}
\item Assume $n$ is even ($n\geq 2$). 
Then, for each $\eta\in V'(\Q)$ with $q(\eta)>0$ and each rational prime $p$, there exists a Laurent polynomial $\widetilde{Q}_{\eta,p}(X_p):=\widetilde{Q}_{\eta,\Phi_p}(X_p)\in \Z[X_p,X^{-1}_p]$ 
depending on $\eta$ and the unramified Schwartz function $\Phi_p={\rm char}_{V'(\Z_p)} \in \mathcal{S}(V'(\Q_p))$ satisfying 
$\widetilde{Q}_{\eta,p}(X_p)=\widetilde{Q}_{\eta,p}(X^{-1}_p)$ such that 
$$a_{E_l}(\eta)=C_{l,n}q(\eta)^{\frac{l-\frac{n}{2}}{2}}\ds\prod_{p}
\widetilde{Q}_{\eta,p}(p^{\frac{l-\frac{n}{2}}{2}})$$
where the constant $C_{l,n}$ is given explicitly in Theorem \ref{FCEisenEven}. 
\item Assume that n is odd, in which case we need a slight modification. 
Write $q(\eta)=\frak d_{\eta}\frak f_\eta^2$ such that $\frak d_\eta$ is the absolute discriminant of $\Bbb Q\left(\sqrt{\epsilon q(\eta)}\right)/\Bbb Q$, where $\epsilon=\begin{cases} 1, &\text{if $q(\eta)_1\equiv 1$ (mod 4)}\\ -1, &\text{if $q(\eta)_1\equiv 3$ (mod 4)}\end{cases}$, and $q(\eta)=2^{v_2(\eta)}q(\eta)_1$, $q(\eta)_1$ odd. Let $\chi_\eta$ be the primitive Dirichlet character corresponding to $\Bbb Q\left(\sqrt{\epsilon q(\eta)}\right)/\Bbb Q$. Then for each rational prime $p$, there exists a Laurent polynomial 
$\widetilde{Q}_{\eta,p}(X_p):=\widetilde{Q}_{\eta,\Phi_p}(X_p)\in \Z[X_p,X^{-1}_p]$ depending on $\eta$ and $\Phi_p={\rm char}_{V'(\Z_p)}$ satisfying 
$\widetilde{Q}_{\eta,p}(X_p)=\widetilde{Q}_{\eta,p}(X^{-1}_p)$ such that
$$a_{E_l}(\eta)=
C_{l,n}' L(\tfrac {n+1}2-l,\chi_\eta) \frak f_\eta^{l-\frac {n}2} \prod_p \widetilde 
 Q_{\eta,p}(p^{l-\frac {n}2}),
$$
where $C_{l,n}'$ is given in Theorem \ref{FCEisenOdd}. 
\end{enumerate}
\end{theorem}

We organize this paper as follows. 
In Section \ref{AFSO}, we recall Pollack's definition of certain types of automorphic forms on $G(\A)$ \cite{Po}. 
In Section \ref{Eisen}, we recall the definition of Eisenstein series, and by assuming the explicit calculation of the Siegel series in 
Section \ref{Siegel}, we write down the formula for the Fourier coefficients of the Eisenstein series. We also compute the Fourier coefficients at rank one index and show that a 
partial sum is related to a vector valued (non-holomorphic) Eisenstein series on $\SL_2(\A)$. 

\smallskip  

\textbf{Acknowledgments.} We would like to thank Tamotsu Ikeda, Hidenori Katsurada, 
Sungmun Cho, Tadashi Miyazaki and Jim Arthur for helpful discussions, and Aaron Pollack and Yi Shan for pointing out some inaccuracies. 
We thank KIAS in Seoul and Waseda University in Tokyo for their incredible hospitality during this research. 

\smallskip

\section{Preliminaries}\label{pre}
For each quadratic space $W$ or its representation matrix $A$, the corresponding symmetric pairing is denoted by $(\ast,\ast)_W$ or $(\ast,\ast)_A$. 

For each integer $n\ge 2$, let $A$ be a positive definite symmetric matrix in 
$M_{n-2}(\Z)$ as defined in Section \ref{intro}. Put  
$$J_{1,n-1}=\left(\begin{array}{ccc}
 0 & 0 & 1 \\
 0 & -A  & 0  \\ 
1 & 0 & 0 
\end{array}\right),\ J_{2,n}=\left(\begin{array}{ccc}
 0 & 0 & 1 \\
 0 & J_{1,n-1}  & 0  \\ 
1 & 0 & 0 
\end{array}\right)
,\ 
 J_{3,n+1}=\left(\begin{array}{ccc}
 0 & 0 & 1 \\
 0 & J_{2,n}  & 0  \\ 
1 & 0 & 0 
\end{array}\right).$$

Let $V'=
\mathbb{G}^{n+2}_a$ be the quadratic space associated to 
$J_{2,n}$. 
Let $V=\mathbb{G}_a e\oplus V'\oplus \mathbb{G}_a f=\mathbb{G}^{n+4}_a$  be the quadratic space associated to 
$J_{3,n+1}$ where 
$$e:=(1,\overbrace{0,\ldots,0}^{n+3}),\ f:=(\overbrace{0,\ldots,0}^{n+3},1)$$
so that $(e,f)_V=\frac{1}{2}eJ_{3,n+1}{}^t f=\frac{1}{2}$. Let $q$ be the quadratic form on $V'$ 
defined by $q(x)=\frac{1}{2}xJ_{2,n} {}^t x$ for $x\in V'$. 
Denote by $$(x,y)=q(x+y)-q(x)-q(y)$$ for $x,y\in V'$ the associated bilinear form. The quadratic form $\tilde q$ on $V$ for $J_{3,n+1}$ 
satisfies $\tilde q(\alpha e+v'+\beta f)=\alpha\beta+q(v')$.

Let
$$G=\SO(V)=\SO(3,n+1)=\{g\in \SL_{n+4}\ |\ {}^t g J_{3,n+1}g=J_{3,n+1}\}$$ and we also consider 
$$\SO(V')=\SO(2,n)=
\{g\in \SL_{n+2}\ |\ {}^t g J_{2,n}g=J_{2,n}\}.
$$

The natural embedding $V'\hookrightarrow V,\ v'\mapsto (0,v',0)$ yields an embedding 
$$\SO(V')\hookrightarrow \SO(V),\ g'\mapsto 
\left(\begin{array}{ccc}
 1 & 0& 0 \\
 0 & g'  & 0  \\ 
0 & 0 & 1 
\end{array}\right).$$ 

Throughout this paper, we assume $G=\SO(V)$ splits at any finite place of $\Q$ as in Section \ref{intro}.

\begin{remark}\label{difPo}
Our definition of $q$ is slightly different from the one in \cite{Po}, which we recall in Section \ref{infty}.  
Then, our $\SO(V')$ is isomorphic over $\R$ to Pollack's one but not over $\Q$ 
since any normalized orthogonal transformation matrices have to involve $\sqrt{2}$. 
\end{remark}

Let $P=MN$ be the Siegel parabolic subgroup, where 

$$M=\{\diag(t,m,t^{-1})\in G\ |\ t\in \GL_1,\ m\in \SO(V')\}\simeq \GL_1\times \SO(V')$$ with 
$\nu:M\lra \GL_1,\ \diag(t,m,t^{-1})\mapsto t$, 
and 
$$N=\left\{n(\textbf{x}):=
\left(\begin{array}{ccc}
 1 & -\textbf{x}J_{2,n} & -\frac{1}{2}\textbf{x}J_{2,n}{}^t\textbf{x} \\
 0 & 1_{n+2}  & {}^t\bf{x}  \\ 
0 & 0 & 1 
\end{array}\right) \Bigg|\ \textbf{x}\in \mathbb{G}^{n+2}_a
\right\}\stackrel{\sim}{\lra}V',\ n(x)\mapsto x.$$
Then, $\diag(t,m,t^{-1})n\in P$ acts on $x\in V'$ by $t^2(m\cdot x)$ where 
 $m\cdot x$ stands for the usual matrix multiplication. Once a suitable Haar measure is chosen, 
the modulus character of $P$ is given by $\delta_P(\diag(t,m,t^{-1}))=|t|^{n+2}$.

\section{Automorphic forms on $G=\SO(3,n+1)$}\label{AFSO}

Recall $G=\SO(V)$. 
The maximal compact subgroup $K$ of $G(\Bbb R)$ is $S({\rm O}(3)\times {\rm O}(n+1))$. 
Let $K_0=\SO(3)\times \SO(n+1)$.  
The projection onto the first factor induces a surjective homomorphism $S({\rm O}(3)\times {\rm O}(n+1))\rightarrow {\rm O}(3)=({\rm SU}(2)/\mu_2)\rtimes \{\pm 1\}$. For each $\lambda\in \C$, we consider the normalized induced representation 
$$I(\lambda):={\rm Ind}^{G(\R)}_{P(\R)} \, |\nu|^{\lambda-\frac{n}{2}-1}.$$

By the Peter-Weyl theorem, we have 
$${\rm Ind}^{K_0}_{K_0\cap P(\R)}\mathbf{1}\simeq 
C^\infty((S^2\times S^{n})/\{\pm 1\})\simeq \bigoplus_{a,b\ge 0\atop a+b:\, {\rm even}}
\mathcal{H}^{a,b},\ 
\mathcal{H}^{a,b}:=\mathcal{H}^a(\C^3)\boxtimes \mathcal{H}^b(\C^{n+1})$$
where $S^k$ stands for the $k$-dimensional sphere and $\mathcal{H}^i(\C^j)$ 
stands for the harmonic polynomials on $\C^j$ of degree $i$. Using this, we have 
\begin{equation}\label{K-type}
I(\lambda)|_K= \bigoplus_{a,b\ge 0\atop a+b:\, {\rm even}}{\rm Ind}^K_{K_0}
\mathcal{H}^{a,b}.
\end{equation}
If the twist of $\mathcal{H}^{a,b}$ by $K/K_0$ is isomorphic to $\mathcal{H}^{a,b}$, then 
$${\rm Ind}^K_{K_0}
\mathcal{H}^{a,b}=\mathcal{H}^{a,b,+}\oplus \mathcal{H}^{a,b,-}$$
where $\mathcal{H}^{a,b,\ve}$ is an extension of $\mathcal{H}^{a,b}$ to $K$ with the 
action of a character $\ve:K/K_0\lra \{\pm 1\}$. Otherwise, ${\rm Ind}^K_{K_0}
\mathcal{H}^{a,b}$ is irreducible. 
 In particular, if $b=0$ (hence $a$ is even, say $a=2l$), then 
$\mathcal{H}^{a,b}$ admits a natural extension to a representation 
$\mathcal{H}^{a,b,+}$ of $K$; by abuse of notation, we denote this extension again by 
$\mathcal{H}^{a,b}$. It appears in $I(\lambda)|_K$ with multiplicity one.

\begin{defin}\label{MFG} Let $l\ge 0$ be an integer. Let 
 $\Bbb V_l={\rm Sym}^{2l}(\Bbb C^2)$ be the $(2l+1)$-dimensional representation of $K$ that factors through ${\rm O}(3)$ and we regard it with 
 $\mathcal{H}^{2l}(\C^3)\boxtimes \mathbf{1}$ as a unique minimal $K$-type of $I(l)$. 
We fix a basis $\{[X]^{l+v}[Y]^{l-v}\}_{v=-l}^l$ of $\Bbb V_l$ where  $[X^k]=\frac {X^k}{k!}$ and  $[Y^k]=\frac {Y^k}{k!}$ for $k\ge 0$. 
Modular forms on $G$ of weight $l$ are $\Bbb V_l$-valued automorphic functions $\phi$ on $G(\Bbb A)$ that satisfy:

(1) $\phi(gk)=k^{-1}\phi(g)$ for all $g\in G(\Bbb A)$ and $k\in K$. 

(2) $\phi$ is annihilated by a special differential operator $\mathcal D_l$. 

(3) As a $(\frak g, K)$-module, $\phi$ generates an irreducible constituent of $I(l+1)$. 
\end{defin}

\begin{theorem}\cite{Po} Suppose $\phi$ is a modular form of weight $l\geq 1$ on $G$. Then 
$$\phi(g)=\phi_0(g)+\sum_{\eta\in V'(\Bbb Q)\atop q(\eta)\geq 0, \eta\neq 0} a_\phi(\eta)(g_{\f})\mathcal W_{2\pi \eta}(g_\infty),
$$
for $g=g_{\f}g_\infty$ in $G(\Bbb A_{\f})\times G(\Bbb R)$, where $a_\phi(\eta): G(\Bbb A_{\f})\longrightarrow \Bbb C$ is a locally constant function. Moreover,
$$\phi_0(m)=\Phi(m)X^{2l}+\beta(m_f)X^lY^l+\Phi'(m)Y^{2l},
$$
where $\Phi$ is an automorphic function associated to a holomorphic modular form of weight $l$ on $M$, $\beta$ is a locally constant function on $M(\Bbb A_{\f})$, and $\Phi'$ is a certain $(K\cap M)$-right translate of $\Phi$.
\end{theorem}

Here $\mathcal W_{\eta}(g): G(\Bbb R)\longrightarrow \Bbb V_l$ is a generalized Whittaker function of type $\eta$ satisfying 
\begin{enumerate}
\item $\mathcal W_\eta(n(x)g)=e^{i(\eta,x)} \mathcal W_\eta(g)$;

\item $\mathcal W_\eta(gk)=k^{-1} \mathcal W_\eta(g)$;

\item $D_l \mathcal W_\eta(g)=0$;

\item Suppose $\eta\in V'(\Bbb R)$, $\eta\ne 0$, and $(\eta,\eta)\geq 0$. For  $t\in GL_1(\Bbb R), m\in \SO(V')(\Bbb R)$, set
\begin{equation}\label{ourueta}
u_\eta(t,m)=it\sqrt{2}(\eta,m(iv_1-v_2))
\end{equation}
where 
$$v_1:=(1,\overbrace{0,\ldots,0}^{n},1),\ v_2:=(0,1,\overbrace{0,\ldots,0}^{n-2},1,0).
$$
 Then
\begin{equation}\label{PollackSpherical}
\mathcal W_\eta(t,m)=t^l |t| \sum_{-l\leq v\leq l} \left(\frac {|u_\eta(t,m)|}{u_\eta(t,m)}\right)^v K_v(|u_\eta(t,m)|) [X^{l+v}][Y^{l-v}],
\end{equation}
where $K_v$ is the $K$-Bessel function defined by
$$K_v(y)=\frac 12\int_0^\infty e^{-y(t+t^{-1})/2} t^v\, \frac {dt}t.
$$
\end{enumerate}

By definition, vectors $\eta\in V'(\Bbb Q)$ with $q(\eta)>0$ are anisotropic vectors and they are called rank 2 elements in \cite{Po}. 

\section{Eisenstein series}\label{Eisen} 
In this section, we follow \cite[Section 4]{Po}.  
For any $l>n+1$ and a Schwartz-function $\Phi_{\f}$ on $V(\A_{\bf f})$, we define a section 
$f$ of ${\rm Ind}^{G(\A)}_{P(\A)}\, |\nu|^{s-\frac {n+2}2}$ (normalized induction) by 
$$f_l(g,\Phi_{\f},s)=f_{{\rm fte}}(g_{\f},\Phi_{\f},s)f_{l,\infty}(g_\infty,s),\ 
g=g_{\f}g_\infty\in G(\A)
$$
where 
$$f_{{\rm fte}}(g_{\f},\Phi_{\f},s)=\int_{\GL_1(\A_{\f})}|t|^s\Phi_{\f}(tg_{\f}^{-1}e)dt,\ g_{\f}\in 
G(\A_{\f})
$$
and $f_{l,\infty}(g_\infty,s)$ is the $V_l={\rm Sym}^{2l}(\C^2)$-valued, $K$-equivalent 
section of ${\rm Ind}^{G(\R)}_{P(\R)}\, |\nu|^{s-\frac {n+2}2}$. See Section \ref{infty} for a precise definition.
Then, we define the associated Eisenstein series as 
$$E_l(g,\Phi_{\f},s):=\sum_{\gamma\in P(\Q)\bs G(\Q)}f(\gamma g,\Phi_{\f},s),\ g\in G(\A).
$$
By definition, $E_l(g,\Phi_{\f},s)=E_l(g_\infty,g_{\f}\cdot \Phi_{\f},s)$ where 
$g_{\f}\cdot \Phi_{\f}$ is defined by 
$g_{\f}\cdot \Phi_{\f}(v)=\Phi_{\f}(g^{-1}_{\f}v)$ for $v\in V(\A)$. It converges absolutely for $Re(s)>n+2$.
When $s=l+1$ and $l+1>n+2$, $E_l(g,\Phi_\f,s=l+1)$ gives rise to a modular form of weight $l$, and it has the Fourier expansion as 
\begin{equation}\label{EFE}
E_l(g,\Phi_{\f})=E_0(g)+\sum_{0\neq \eta\in V'(\Q)\atop q(\eta)\ge 0}
a_{E_l}(\eta)(g_{\f}\cdot \Phi_{\f})\mathcal{W}_{2\pi \eta}(g_\infty)
\end{equation}
where $E_0(g)$ is the constant term, and $\mathcal{W}_{2\pi \eta}(g_\infty)$ is the Pollack's spherical function.
Now we assume $\Phi_{\f}=\otimes_p' \Phi_p\otimes \Phi_\infty$, and $\Phi_p$ is unramified for all $p$ (hence, $\Phi_{\f}$ is the characteristic function of $V'(\widehat{\Z})$). By the strong approximation, 
we have 
\begin{equation}\label{EFE1}
E_l(g,\Phi_{\f})=E_0(g)+\sum_{0\neq \eta\in V'(\Q)\atop q(\eta)\ge 0}
a_{E_l}(\eta)(\Phi_{\f})\mathcal{W}_{2\pi \eta}(g_\infty)
\end{equation}
for $g=\gamma k_\f g_\infty \in G(\A)=G(\Q)(G(\widehat{\Z})\times G(\R))$, since 
$k_\f\cdot \Phi_\f=\Phi_\f$.

We compute explicitly $a_{E_l}(\eta):=a_{E_l}(\eta)(\Phi_{\f})$.

\subsection{Rank 2 Fourier coefficients}
If $\eta$ is of rank 2, i.e., $q(\eta)\ne 0$, 
\begin{equation}\label{FCSS}
a_{E_l}(\eta)(\Phi_{\f})\mathcal{W}_{2\pi \eta}(g_\infty)=J(s,\eta,f_l)=\int_{V'(\Bbb A)} 
 \psi((\eta,x)) f_l(w n(x), \Phi_\f,s)\, dx.
\end{equation}
where $\psi=\otimes'_p\psi_p:\Q\bs \A\lra \C^\times$ is the standard additive character. 
Write
$$J(s,\eta,f_l)=\prod_p J(s,\eta,\Phi_p),\ J(s,\eta,\Phi_p):=
\int_{V'(\Q_p)}\psi_p((\eta,x)) f_l(w n(x), \Phi_\f,s)\, dx
$$
where $w\in G=\SO(V)$ is defined by $w\cdot e=f$, $w\cdot f=e$, and $w|_{V'}={\rm id}_{V'}$.  
Notice that $J(s,\eta,\Phi_p)$ is also defined for $\eta\in V'(\Q_p)$ with $q(\eta)\neq 0$.  
Further, by \cite[p.633, line -6]{Po}, $J(s,\eta,\Phi_p)=
\ds\int_{\GL_1(\Q_p)}\int_{V'(\Q_p)}\psi_p((\eta,x))\Phi_p(t,tx,-tq'(x)))|t|^s_p\, dxdt$. 
Since $t$ runs over $\Z^\times_p$, $\Phi_p(t,tx,-tq'(x)))$ is preserved under the transformation 
$x\mapsto x+x_0$ for any $x_0\in V'(\Z_p)$. By the change of variables, 
$J(s,\eta,\Phi_p)=\psi((\eta,x_0))J(s,\eta,\Phi_p)$ for any $x_0\in V'(\Z_p)$. Thus, 
$J(s,\eta,\Phi_p)=0$ if $\eta\not\in V'(\Z_p)$. 

 
\subsubsection{$p$-adic computation: Siegel series}\label{p-adic}
Let $\psi:\Q^\times_p\lra \C^\times$ be the standard additive character and henceforth 
we simply write it by $\psi(x)=e^{2\pi i x}$ (though $x$ should be the fractional 
part of $x$ to be more precise). 
For $\Phi_p={\rm char}_{V'(\Z_p)}$ and $\eta\in V'(\Q_p)$ with $q(\eta)\neq 0$, we will compute the Siegel series: 
$$J(s,\eta,\Phi_p)=\sum_{r=0}^\infty p^{-rs} \int_{p^{-r}V'(\Bbb Z_p)} \psi((\eta,x)) \text{char}(p^r q(x)\in\Bbb Z_p)\, dx.
$$ 

We denote the integral by $B_{r,\eta}$. 
Then by the change of variables,
$$B_{r,\eta}=\int_{p^{-r}V'(\Bbb Z_p)} \psi((\eta,x)) \text{char}(p^r q(x)\in\Bbb Z_p)\, dx
=p^{r(n+2)}\int_{V'(\Bbb Z_p)\atop q(x)\equiv 0 \, \text{mod $p^r$}} \psi\Big(\frac {(\eta,x)}{p^r}\Big)\, dx.
$$

Set $x=u+y$, where $u\in V'(\Bbb Z_p/p^r\Bbb Z_p)$, and $y\in V'(p^r\Bbb Z_p)$. If $y\equiv 0$ (mod $p^r$), $q(y)\equiv 0$ (mod $p^r$). It follows from this that, upon making the change of variables in $y$, the factor $p^{-r(n+2)}$ arises, cancelling the factor $p^{r(n+2)}$ obtained earlier. Thus  
\begin{equation}\label{main}
B_{r,\eta}=\sum_{u\in V'(\Bbb Z/p^r\Bbb Z)\atop q(u)\equiv 0\, \text{(mod $p^r$)}} \psi\Big(\frac {(\eta,u)}{p^r}\Big).
\end{equation}

When $r=0$, we have the integral
$$\int_{V'(\Bbb Z_p)} \psi((\eta,x))\, dx.
$$
Then it is $\begin{cases} 1, &\text{if $\eta\in V'(\Bbb Z_p)$}\\ 0, &\text{otherwise}\end{cases}$.

If $r\geq 1$, we need to compute
$$
\sum_{u\in V'(\Bbb Z/p^r\Bbb Z)\atop q(u)\equiv 0\, \text{(mod $p^r$)}} \psi\Big(\frac {(\eta,u)}p\Big).
$$

We need to divide into three cases:

\begin{theorem}\label{Siegel-even} Suppose $n$ is even and let ${\rm dim}V'=n+2=2m$. Let $p^{k}\|\eta$, and $p^{2k+k'}\|q(\eta)$. Then $$J(s,\eta,\Phi_p)=(1-p^{m-1-s})Q_{\eta,p}(p^{m-s}),$$ where $Q_{\eta,p}(X)\in \Bbb Z[X]$ is a polynomial of degree $2k+k'$, and it satisfies the functional equation 
$$
X^{2k+k'}Q_{\eta,p}(X^{-1})=Q_{\eta,p}(X).
$$
\end{theorem}
Let
\begin{equation}\label{wQeven}
\widetilde Q_{\eta,p}(X)=X^{2k+k'}Q_{\eta,p}(X^{-2}).
\end{equation}
Then $\widetilde Q_{\eta,p}(X^{-1})=\widetilde Q_{\eta,p}(X)$.

Special case: $k=0, k'=0$. Then $J(s,\eta,\Phi_p)=1-p^{m-1-s}$.
If $k=0$, $k'=1$, $J(s,\eta,\Phi_p)=1+(p^m-p^{m-1})p^{-s}-p^{2m-1-2s}=(1-p^{m-1-s})(1+p^{m-s})$. Hence $Q_{\eta,p}(p^{m-s})=1+p^{m-s}$.
More generally, if $k=0$, $Q_{\eta,p}(p^{m-s})=\sum_{a=0}^{k'} (p^{m-s})^a$.

\begin{theorem}\label{Siegel-odd} Let $n$ be odd, 
${\rm dim}V'=n+2=2m+1$, and $p>2$. Let $p^k\|\eta$ and $p^{2k+k'}\|q(\eta)$, and $q(\eta)_1=\frac {q(\eta)}{p^{2k+k'}}$. 
Let $\chi_\eta(p)=\begin{cases} (\frac {q(\eta)_1}p), &\text{if $2k+k'$ even}\\ 0, &\text{if $2k+k'$ odd}\end{cases}$. Then
$$J(s,\eta,\Phi_p)=\frac {1-p^{2(m-s)}}{1-\chi_{\eta}(p) p^{m-s}} Q_{\eta,p}(p^{m-s}),
$$ 
where $Q_{\eta,p}(X)\in \Bbb Z[X]$ is a polynomial of degree $2\lfloor\frac {2k+k'}2\rfloor$, and it satisfies the functional equation 
$$
p^lX^{2l}Q_{\eta,p}(p^{-1}X^{-1})=Q_{\eta,p}(X),
$$
where $2k+k'=2l$ or $2l+1$.
\end{theorem}


Let 
\begin{equation}\label{wQodd}
\widetilde{Q}_{\eta,p}(X)=X^{-l} Q_{\eta,p}(p^{-\frac 12}X).
\end{equation}
Then $\widetilde{Q}_{\eta,p}(X^{-1})=\widetilde{Q}_{\eta,p}(X)$.

Special case: $k=0, k'=0$. Then $J(s,\eta,\Phi_p)=1+\Big(\frac {q(\eta)}p\Big) p^{m-s}$. 
If $k=0$, $k'=1$, $J(s,\eta,\Phi_p)=1-p^{2(m-s)}$, and so $Q_{\eta,p}(p^{m-s})=1$.
If $k=0$, $k'=2$, $Q_{\eta,p}(p^{m-s})=1-\Big(\frac {q(\eta)_1}p\Big) p^{m-s}+p^{2(m-s)+1}$.

More generally, if $k=0$, $k'=2l$ even, (let $\varepsilon=\Big(\frac {q(\eta)_1}p\Big)$),
Then 
$$J(s,\eta,\Phi_p)=1+\sum_{r=1}^l p^{2r(m-s)}(p^r-p^{r-1})+\varepsilon p^{(2l+1)(m-s)+l}=(1+\varepsilon p^{m-s})Q_{\eta,p}(p^{m-s}),
$$
where $Q_{\eta,p}(p^{m-s})=\sum_{i=0}^l p^{2i(m-s)+i}-\varepsilon\sum_{i=1}^{l} p^{(2i-1)(m-s)+i-1}.$

If $k=0$, $k'=2l+1$ odd ($l\geq 1$),
$$J(s,\eta,\Phi_p)=1+\sum_{r=1}^l p^{2r(m-s)}(p^r-p^{r-1})-p^{(2l+2)(m-s)+l}=(1- p^{2(m-s)})Q_{\eta,p}(p^{m-s}),
$$
where $Q_{\eta,p}(p^{m-s})=1+\sum_{i=1}^{l} p^{2i(m-s)+i}.$

Recall that $(\tfrac n2)=0$ if $2|n$ or $n\equiv 3$ (mod 4). For $n\equiv 1$ (mod 4),
$(\tfrac n2)=\begin{cases} 1, &\text{if $n\equiv 1$ (mod 8)}\\ -1, &\text{if $n\equiv 5$ (mod 8)}\end{cases}$. 

\begin{theorem}\label{Siegel-2} Let $n$ be odd, ${\rm dim}V'=n+2=2m+1$, and $p=2$. Let $2^k\|\eta$ and $2^{2k+k'}\|q(\eta)$. Let $l=\lfloor \frac {2k+k'}2\rfloor$. 
Let $q(\eta)_1=q(\eta)2^{-2k-k'}$. 
Let $\chi_\eta(2) = \begin{cases} (\frac {q(\eta)_1}2), &\text{if $2k+k'$ even}\\ 0, &\text{if $2k+k'$ odd}\end{cases}$.
Then
$$J(s,\eta,\Phi_2)=\frac {1-2^{2(m-s)}}{1-\chi_{\eta}(2) 2^{m-s}} Q_{\eta,p}(2^{m-s}),
$$ 
where $Q_{\eta,p}(X)\in \Bbb Z[X]$ is a polynomial of degree $\alpha$, and $\alpha=\begin{cases} 2l, &\text{if $\chi_\eta(2)=0$}\\ 2l+2, &\text{if $\chi_\eta(2)=\pm 1$}\end{cases}$. And
$Q_{\eta,p}(X)$ satisfies the functional equation 
$$
2^{\frac {\alpha}2} X^{\alpha}Q_{\eta,p}(2^{-1}X^{-1})=Q_{\eta,p}(X).
$$
\end{theorem}

Since the proofs are very long, we will postpone the proofs until Section \ref{Siegel}.

\subsubsection{Archimedean computation}\label{infty}

For the archimedean place, recall some definitions. The $K$-Bessel function $K_v(y)$ is defined as
$$K_v(y)=\frac 12\int_0^\infty e^{-y(t+t^{-1})/2} t^v \frac {dt}t.
$$
It satisfies $K_{-v}(y)=K_v(y)$.
The maximal compact subgroup $K$ of $\SO(V)(\Bbb R)$ is $S({\rm O}(3)\times {\rm O}(n+1))$. Denote by $\Bbb V_l=Sym^{2l}(\Bbb C^2)$, the $(2l+1)$-dimensional representation of $K$ that factors through ${\rm O}(3)$, and let $\{X^{2l}, X^{2l-1}Y,...,Y^{2l}\}$ be a basis of $\Bbb V_l$.

Recall the notations in \cite{Po}: $V'=V_2\oplus V_n$ with the quadratic form $q(x,y)=q_2(x)-q_n(y)$, where $(V_2,q_2)$ and $(V_n,q_n)$ are 2-dimensional ($n$-dimensional, resp.) quadratic space with positive definite quadratic forms. Let $\{v_1,v_2\}$ be an orthonormal basis of $V_2(\Bbb R)$. Note that this quadratic form is different from the one in Section \ref{pre}. It accounts for the appearance of $\sqrt{2}$.

Let $f_l(g,s)$ be the $\Bbb V_l$-valued section. For $x\in V'(\Bbb R)$, let $x=(x_2,x_n)$, where $x_2\in V_2, x_n\in V_n$. 
 Then
$$f_l(w n(x),s)=\frac {p_{V_3}(1,x,-q(x))^l}{||(1,x,-q(x))||^{s+l}},
$$
where $p_{V_3}: V=V_3\oplus V_{n+1}\longrightarrow V_3$ is the orthogonal projection. Then
$$p_{V_3}(1,x,-q(x))=-\frac 1{2\sqrt{2}} \left( (\sqrt{2} x_2, iv_1+v_2)X^2+(||x_2||^2-||x_n||^2-2)XY+(\sqrt{2} x_2, iv_1-v_2)Y^2\right)
$$
and 
$$||(1,x,-q(x))||^2=\tau(\sqrt{2} x_n, \sqrt{2} x_2)^2=||x_2||^2+\left( 1+\frac {||x_n||^2-||x_2||^2}4 \right)^2.
$$

Consider, for $\eta\in V'(\Bbb R)$ with $q(\eta)>0$, 
$$J(s,f_l)=\int_{V'(\Bbb R)} e^{-2\pi i(\eta,x)} f_l(w n(x),s)\, dx.
$$

Let $a=-\sqrt{2} (x_2,iv_1+v_2)$ and $b=\frac {||x_2||^2-||x_n||^2-2}2$. Then we have

\begin{lemma}\cite{Po} If $b>0$,
\begin{eqnarray*}
&&  \left(aX^2+2bXY-\bar aY^2\right)^l=\sum_{v=0}^l \begin{pmatrix} l\\v\end{pmatrix} 2^{l-v} \delta_{v0}^{\frac 12} (XY)^{l-v} (|a|^2+b^2)^{\frac {l-v}2} \\
&& \phantom{xxxxxxxxxxxxx}\cdot {}_2 F_1\left(\frac {v-l}2,\frac {v+l+1}2;v+1,\frac {|a|^2}{|a|^2+b^2}\right)
\left(a^v X^{2v}+(-\bar a)^vY^{2v}\right),
\end{eqnarray*}
where $\delta_{v0}^{\frac 12}=\begin{cases} \frac 12, &\text{if $v=0$}\\ 1, &\text{otherwise}\end{cases}$.
\end{lemma}

Here we use the identity
$${}_2F_1(a',b';c',z)=(1-z)^{-a'} {}_2F_1(a',c'-b';c',\frac z{z-1}).
$$
So
$$
(|a|^2+b^2)^{\frac {l-v}2} {}_2 F_1\left(\frac {v-l}2,\frac {v+l+1}2;v+1,\frac {|a|^2}{|a|^2+b^2}\right)
=b^{l-v} {}_2 F_1\left(\frac {v-l}2,\frac {v-l}2+\frac 12;v+1,-\frac {|a|^2}{b^2}\right).
$$

So 
$$J(s,f_l)=(2\sqrt{2})^{-l} \int_{V'(\Bbb R)} e^{-2\pi i(\eta,x)} \tau(\sqrt{2}x_n,\sqrt{2}x_2)^{-s-l} \left(aX^2+2bXY-\bar aY^2\right)^l
\, dx.
$$

Note that $|a|^2=2||x_2||^2$ and $|a|^2+b^2=\tau(\sqrt{2}x_n,\sqrt{2}x_2)^2$.
Let

\begin{eqnarray*}
&& \Phi_v(s)=\int_{V'(\Bbb R)} e^{-2\pi i (\eta,x)} \tau(\sqrt{2}x_n,\sqrt{2}x_2)^{-s-v} a^v
{}_2 F_1\left(\frac {v-l}2,\frac {l+v+1}2;v+1,\frac {|a|^2}{|a|^2+b^2}\right)\, dx \\
&& \phantom{xxxx} =\int_{V'(\Bbb R)} e^{-2\pi i (\eta,x)} (|a|^2+b^2)^{-\frac {s+l}2} a^v b^{l-v}
{}_2 F_1\left(\frac {v-l}2,\frac {v-l}2+\frac 12;v+1,-\frac {|a|^2}{b^2}\right)\, dx.
\end{eqnarray*}

Let 
$$\widetilde\Phi_v(s)=\int_{V'(\Bbb R)} e^{-2\pi i (\eta,x)} \tau(\sqrt{2}x_n,\sqrt{2}x_2)^{-s-v} (-\bar a)^v
{}_2 F_1\left(\frac {v-l}2,\frac {l+v+1}2;v+1,\frac {|a|^2}{|a|^2+b^2}\right)\, dx \\
$$
Then $\Phi_0(s)=\widetilde\Phi_0(s)$, and

\begin{eqnarray*}
&& J(s,f_l)=(2\sqrt{2})^{-l}\sum_{v=1}^l \begin{pmatrix} l\\v\end{pmatrix} 2^{l-v} 
\Phi_v(s) X^{l+v}Y^{l-v} +
(2\sqrt{2})^{-l} \sum_{v=1}^l \begin{pmatrix} l\\v\end{pmatrix} 2^{l-v} 
\tilde\Phi_v(s) X^{l-v}Y^{l+v} \\
&&\phantom{xxxxxxx} +(2\sqrt{2})^{-l} 2^l \Phi_0(s)X^lY^l.
\end{eqnarray*}

Pollack \cite{Po} proved the following when $s=l+1$,
\begin{theorem} Let $v\geq 0$.
\begin{eqnarray*}
&&  \Phi_v(l+1)= \frac {(2\pi)^{-\frac {n+2}2}2^{-(l-v-1-\frac n2)}\Gamma(v+1)}{\Gamma(l+v+1)\Gamma(l-\frac n2+1)}q(\eta)^{l-\frac n2} \left( \frac {|(2\pi\eta,v_1+iv_2)|}{(2\pi\eta,v_1+iv_2)}\right)^v K_v(\sqrt{2}|(2\pi\eta,v_1+iv_2)|).
\end{eqnarray*}
\end{theorem}

We see that since $K_{-v}(x)=K_v(x)$,
$$\widetilde\Phi_v(l+1) = \frac {(2\pi)^{-\frac {n+2}2}2^{-(l-v-1-\frac n2)}\Gamma(v+1)}{\Gamma(l+v+1)\Gamma(l-\frac n2+1)}q(\eta)^{l-\frac n2} \left( \frac {|(2\pi\eta,v_1+iv_2)|}{\overline{(2\pi\eta,v_1+iv_2)}}\right)^v K_v(\sqrt{2}|(2\pi\eta,v_1+iv_2)|).
$$
Hence 
\begin{theorem}\cite{Po}
$$J(l+1,f_l)=\pi^{2l+1-\frac n2} q(\eta)^{l-\frac n2} \mathcal W_{2\pi\eta}(1),
$$
where 
for $\eta\in V'(\Bbb R)$, and $t\in GL_1(\Bbb R), m\in SO(V')(\Bbb R)$, 
$$u_\eta(t,m)=\sqrt{2} ti(\eta,m(iv_1-v_2)),
$$
$$\mathcal W_{2\pi \eta}(t,m)=t^l |t| \sum_{v=-l}^l \left(\frac {|u_\eta(t,m)|}{u_\eta(t,m)}\right)^v K_v(|u_\eta(t,m)|) X^{l-v}Y^{l+v}.
$$
\end{theorem}

We can write

\begin{eqnarray*}
&& \mathcal W_{2\pi \eta}(t,m)=t^l |t| \sum_{v=1}^l \left(\frac {|u_\eta(t,m)|}{u_\eta(t,m)}\right)^v K_v(|u_\eta(t,m)|) X^{l-v}Y^{l+v} \\
&& \phantom{xxxxx} + t^l |t| \sum_{v=1}^l \left(\frac {|u_\eta(t,m)|}{\overline{u_\eta(t,m)}}\right)^v K_v(|u_\eta(t,m)|) X^{l+v}Y^{l-v}
+ t^l|t| K_0(|u_\eta(t,m)|)X^lY^l.
\end{eqnarray*}

After combining $p$-adic and archimedean computations, we have

\begin{theorem}[$n$ even]\label{FCEisenEven}{\rm(}The rank 2 case{\rm)}
For $\eta\in V'(\Q)$ with $q(\eta)>0$, we have
$$a_E(\eta)=C_{l,n} \prod_p Q_{\eta,p}(p^{l-\frac n2})=C_{l,n} q(\eta)^{\frac {l-\frac n2}2}\prod_p \widetilde Q_{\eta,p}(p^{\frac {l-\frac n2}2})
$$
where $C_{l,n}=\ds\frac {\pi^{2l+1-\frac n2}}{\zeta(l+1-\frac n2)}$. 
\end{theorem}

One corollary of our explicit computation of the Siegel series is that $a_E(\eta)$ have the bounded denominator since $Q_{\eta,p}(X)\in\Bbb Z[X]$.

For $n$ odd, we need some modification. Write $q(\eta)=\frak d_{\eta}\frak f_\eta^2$ such that $\frak d_\eta$ is the absolute discriminant of 
$\Bbb Q\left(\sqrt{\epsilon q(\eta)}\right)/\Bbb Q$ as in the introduction. Let $\chi_\eta$ be the primitive Dirichlet character corresponding to $\Bbb Q\left(\sqrt{\epsilon q(\eta)}\right)/\Bbb Q$. 
We also need the following the functional equation of Dirichlet $L$-function:
$$
L(l-\tfrac {n-1}2,\chi_\eta)=(-1)^{l+n} (2\pi)^{l-\frac {n+1}2} \pi \frac {\frak d_{\eta}^{\frac n2-l}}{(l-\frac {n+1}2)!} L(\tfrac {n+1}2-l,\chi_\eta).
$$

 \begin{theorem}[$n$ odd]\label{FCEisenOdd} For rank 2 case, namely, $q(\eta)>0$, we have 
\begin{eqnarray*}
&& a_E(\eta)=\frac {L(l-\frac {n-1}2,\chi_\eta)}{\zeta(2l-n+1)} \pi^{2l+1-\frac n2} q(\eta)^{l-\frac n2} \prod_{p} Q_{\eta,p}(p^{\frac {n-1}2-l}) \\
 &&\phantom{xxxx} =C_{l,n}' L(\tfrac {n+1}2-l,\chi_\eta) \frak f_\eta^{l-\frac {n}2}\prod_p \widetilde 
 Q_{\eta,p}(p^{l-\frac {n}2}),
\end{eqnarray*} 
where $C_{l,n}'=\frac {(-1)^{\frac {3n+1}2} 2^{l-\frac {n-1}2} \pi^{l+\frac 12}(2l-n+1)!}{B_{2l-n+1}(l-\frac {n+1}2)!}$ and $B_{2l-n+1}$ is the Bernoulli number.
 \end{theorem}

We can write $a_E(\eta)=(*)\pi^{l+\frac 12}L(\frac {n-1}2-l,\chi_\eta) \prod_{p} Q_{\eta,p}(p^{l-\frac {n+1}2})$, where $(*)$ is a rational number. Since $L(\frac {n-1}2-l,\chi_\eta)\in\Bbb Q$ and $Q_{\eta,p}(X)\in\Bbb Z[X]$, $a_E(\eta)$ have the bounded denominator.

\subsection{Rank 1 Fourier coefficients}

Let $\eta$ be a rank one element of $V'(\Z)$, i.e., $\eta\ne 0$ and $q(\eta)=0$. Given $\eta\in V'(\Bbb Z)$, define $v_p(\eta)=\max \{n : \, \eta\in p^n V'(\Bbb Z_p)\}$, and define $\sigma_l(\eta)=\prod_p \sigma_{l,p}(\eta)$, where
$\displaystyle\sigma_{l,p}(\eta)=\sum_{i=0}^{v_p(\eta)} p^{il}$. 

\begin{theorem}\label{mainRankone} For each rank one element $\eta\in V'(\Z)$, 
the Fourier coefficient of $E_l(g,\Phi_\f)$ at $\eta$ is
$$2 i^l l! (2\pi)^l \sigma_l(\eta) \mathcal W_{2\pi\eta}(g_\infty).
$$
\end{theorem}

We follow \cite[Section 3.4]{Po1} as suggested by \cite[Section 4.3]{Po}.

\subsubsection{$p$-adic computation}\label{non-archicomp}

For each $p$-adic field, we need the integral
\begin{equation}\label{rank1}
\int_{\eta^\perp(\Bbb Q_p)\backslash V'(\Bbb Q_p)} \psi(\eta,x) f_p(\gamma_\eta n(x)g_p,\Phi_f,s=l+1)\, dx.
\end{equation}
where $\gamma_\eta\in G(\Bbb Q)$ is such that $\gamma_\eta \eta=e$.
Now we use 
$$f_p(g_p,\Phi_f,s)=\int_{\Bbb Q_p^\times} |t|^s \Phi_p(t g_p^{-1}e)\, dt.
$$

Now $n(x)^{-1}\gamma_\eta^{-1}e=n(-x)\eta=\eta-(x,\eta)e$. 
Hence if $g_p=1$, then
$$(\ref{rank1})=\int_{\Bbb Q_p\times} \int_{\eta^\perp(\Bbb Q_p)\backslash V'(\Bbb Q_p)} \psi(\eta,x)|t|^{l+1}\Phi_p(t\eta-t(x,\eta)e) dxdt.
$$
Now, $\eta^{\perp}(\Bbb Q_p)\backslash V'(\Bbb Q_p)$ is identified with $\Bbb Q_p$ by $x\mapsto (x,\eta)$. Hence
$$(\ref{rank1})=\int_{\Bbb Q_p\times} \int_{\Bbb Q_p} \psi(y)|t|^{l+1}\Phi_p(t\eta-tye) dydt.
$$
By the change of variables $y\mapsto y/t$, it becomes
$$\int_{\Bbb Q_p\times} \int_{\Bbb Q_p} \psi(y/t)|t|^{l}\Phi_p(t\eta-ye) dydt.
$$
Now $t\eta+ye=(y,t\eta,0)\in V(\Bbb Q_p)$, and since $\Phi_p$ is the characteristic function of $\Bbb Z_p e\oplus V'(\Bbb Z_p)\oplus \Bbb Z_p f$, 
$$(\ref{rank1})=\sum_{r\geq -v_p(\eta)} p^{-il} \int_{\Bbb Q_p} \psi(y/p^r)\, dy=\sum_{-v_p(\eta)\leq r\leq 0} p^{-il}=\sigma_{l,p}(\eta).
$$

\subsubsection{Archimedean computation}\label{archicomp}

Let $\psi:\R\lra \C^\times,\ x\mapsto e^{2\pi i x}$ be the standard additive character. 
For each $a\in \R^\times$, put $\psi_a:=\psi(a\ast)$ for simplicity. 
For archimedean place, we need the integral
\begin{equation}\label{rank1-a}
I_\psi(g,\eta,\gamma_\eta)=\int_{\eta^\perp(\Bbb R)\backslash V'(\Bbb R)} e^{-2\pi i(\eta,x)} f_l(\gamma_\eta n(x)g,s=l+1)\, dx,
\end{equation}
for $\gamma_\eta\in G(\Bbb Q)$ and $g\in G(\Bbb R)$. We compute it for $g=1$.
Now $n(x)^{-1}\gamma_\eta^{-1}e=n(-x)\eta=\eta-(x,\eta)e=v$. We identify $\eta^{\perp}(\Bbb R)\backslash V'(\Bbb R)$ is identified with $\Bbb R$ by $x\mapsto (x,\eta)=z$. Now let $\gamma_\eta n(x)=p(z)k(z)$, where $p(z)\in P(\Bbb R), k(z)\in K$. Then since $P$ is the stablizer of $e$,
$p(z)e=a(z)e$, for some $a(z)=\nu(p(z))\in\Bbb C$. Hence $k(z)^{-1}p(z)^{-1}e=a(z)^{-1}k(z)^{-1}e=v.$ 
Here we write $v=(-z,\eta,0)$. If $\eta=\eta_2+\eta_n$, where $\eta_2\in V_2, 
\eta_n\in V_n$, then $\|v\|^2=|z|^2+2\|\eta_2\|^2$. Hence $|a(z)|^{-1}=\|v\|$.

Now 
$$f_l(p(z)k(z),s=l+1)=k(z)^{-1}f_l(p(z),s=l+1)=a(z)^{l+1}p_{V_3}(k(z)^{-1}X^lY^l).
$$

Since $k(z)^{-1}e=a(z)v$, 
$$p_{V_3}(k(z)^{-1}e)=\frac 1{\sqrt{2}} p_{V_3}\left(k(z)^{-1}\left(\frac {u_+}{\sqrt{2}}\right)\right)=a(z)p_{V_3}(v),
$$ 
where $p_{V_3}: V=V_3\oplus V_{n+1}\longrightarrow V_3$ is the orthogonal projection. Now under the correspondence, $X^2\to iv_1-v_2, XY\to \frac {u_+}{\sqrt{2}}, Y^2\to iv_1+v_2$. Since $v=\eta-ze=(-\frac z2 u_+ +\eta_2)+(-\frac z2 u_- +\eta_n)$,
$$p_{V_3}(v)=-\frac z2 u_+ +\eta_2=-\frac 12 (A X^2+BXY+CY^2),
$$
where $A=(\eta_2,iv_1+v_2)$, $C=-\bar A$, and $B=\sqrt{2} z$. Then $|A|^2=\|\eta_2\|^2$.
Hence $p_{V_3}(k(z)^{-1}(XY))=\sqrt{2} a(z) p_{V_3}(v)$. Therefore,

$$f_l(p(z)k(z),s=l+1)=2^{-\frac l2}a(z)^{2l+1} (A X^2+BXY+CY^2)^l.
$$

So
$$I_\psi(1,\eta,\gamma_\eta)=2^{-\frac l2}\int_{-\infty}^\infty e^{-2\pi iz} \frac {(AX^2-\bar AY^2+\sqrt{2}zXY)^l}{(z^2+2|A|^2)^{l+\frac 12}}\, dz.
$$

By \cite[Theorem 3.2.4]{Po}, $I_\psi(1,\eta,\gamma_\eta)$ should be a constant multiple of $\mathcal W_{2\pi\eta}(1)$, where

\begin{equation}\label{sphericalW}
\mathcal W_{2\pi\eta}(1)=\sum_{-l\leq v\leq l} i^v\left(\frac A{|A|}\right)^v K_v(2\pi\sqrt{2}|A|)[X^{l+v}][Y^{l-v}].
\end{equation}

Now the coefficient of $[X^{2l}]$ in $(\ref{rank1-a})$ is

\begin{eqnarray*}
&& (2l)! 2^{-\frac l2} A^l\int_{-\infty}^\infty \frac {e^{-2\pi iz}}{(z^2+2|A|^2)^{l+\frac 12}}\, dz
=A^l\frac {2\sqrt{\pi}}{\Gamma(l+1/2)} \pi^l 2^{-l} |A|^{-l} K_l(2\pi\sqrt{2}|A|) \\
&&\phantom{xxxxxxxxxxxxxxxxsssssx} =(2l)! \frac {2\sqrt{\pi}}{\Gamma(l+\frac 12)} \left(\frac {\pi}2\right)^l \left(\frac A{|A|}\right)^{l} K_{l}(2\pi\sqrt{2}|A|).
\end{eqnarray*}
Therefore, 
$$I_\psi(1,\eta,\gamma_\eta)=2 i^l l! (2\pi)^l \mathcal W_{2\pi\eta}(1).
$$

To see that it is correct, let's compute the coefficient of $[X^{2l-1}Y]$ in $(\ref{rank1-a})$. It is

\begin{eqnarray*}
(2l-1)! 2^{-\frac l2} A^{l-1}l\sqrt{2}\int_{-\infty}^\infty \frac {e^{-2\pi iz}z}{(z^2+2|A|^2)^{l+\frac 12}}\, dz
=(-i)(2l)! \frac {2\sqrt{\pi}}{\Gamma(l+\frac 12)} \left(\frac {\pi}2\right)^l \left(\frac A{|A|}\right)^{l-1} K_{l-1}(2\pi\sqrt{2}|A|).
\end{eqnarray*}

\subsubsection{A partial sum of rank one Fourier coefficients}
Let $Q=LU$ be another maximal parabolic $Q$ such that $L\simeq \GL_2\times \SO(V'')$ 
(see \cite[Section 2.2]{KY3}). 
Let $H$ be the subgroup of $L$ corresponding to $\SL_2$ inside the $\GL_2$-factor of $L$. 
Let $B_H$ be the Borel subgroup of $H$ corresponding to the upper Borel subgroup of 
$\SL_2$ and $\SO_H(2)$ be the maximal compact subgroup of $H(\R)$.  
Let $e':=(1,0,\ldots,0)\in V'$ and $w'\in H(\Q)$ be the Weyl element.  
Recall the section $f=\otimes'_p f_p\otimes f_l(\ast,s=l+1)$ of 
${\rm Ind}^{G(\A)}_{P(\A)}|\nu|^{l-\frac{n}{2}}$. 
It is easy to see that 
$f_{l,H}:=f|_{H(\A)}=f_{l,H,\f}\otimes f_{l,H,\infty}$ belongs to ${\rm Ind}^{H(\A)}_{B_H(\A)}|\cdot |^l$ 
(the irreducible normalized induction) such that $f_{l,H,\infty}$ has the $\SO_H(2)$-type  
$\{v\ |\ -l\le v\le l\}\cap 2\Z$ (note that $e^{i\theta}\in U(1)\simeq \SO_H(2)\subset \SO(3)$ 
goes to $\diag(e^{\frac{i\theta}{2}},e^{\frac{-i\theta}{2}})$ under $\SO(3)\simeq\SU(2)/\mu_2$) and $f_{l,H,\f}$ is a spherical section of 
${\rm Ind}^{H(\A_\f)}_{B_H(\A_\f)}|\cdot |^l$. Let $E(h,f_{l,H};s=l)$ be the Eisenstein series on $H(\A)$ associated to the section $f_{l,H}$ at $s=l$. By automorphy and the strong approximation theorem, $E(h,f_{l,H};s=l)=E(h_\infty,f_{l,H};s=l)$ for $h=h_\f h_\infty H(\A)$ since 
$E(h,f_{l,H};s=l)$ is right $H(\widehat{\Z})$-invariant. Thus, the constant term and the non-constant term of $E(h,f_{l,H};s=l)$ are unique. 

\begin{theorem}\label{partialSum} Up to a non-zero constant multiple, the partial sum 
$\ds\sum_{0\neq n\in \Z}\sigma_l(n)\mathcal W_{2\pi n e'}(h_\infty),\ h_\infty\in H(\R)$ 
is the non-constant part of $E(h_\infty,f_{l,H};s=l)$. 
\end{theorem}
\begin{proof} For $\eta=n e'$ with a non-zero integer $n$, the 
element $\gamma_\eta=w'$ works to get the equation (4) of \cite[p.631, line -3 to -2]{Po}. 
By the results in Section \ref{non-archicomp} and \ref{archicomp},  
$\ds\sum_{0\neq n\in \Z}\sigma_l(n)\mathcal W_{2\pi n e'}(h_\infty),\ h_\infty\in H(\R)$ 
is equal to 
$$\int_{\eta^\perp(\A)\bs V'(\A)}\overline{\psi((\eta,x))}f(\gamma_\eta n(x)h_\infty)dx
=\int_{\A}\overline{\psi_n(t)}f_{l,H}(w' n_H(t)h_\infty)dt
$$
where $n_H(t)$ is the coordinate of the unipotent radical of $B_H$. 
A standard argument shows it is the $\psi_n$-th Fourier coefficient of 
$E(h_\infty,f_{l,H};s=l)$. Note that for $a\in \Q$, the $\psi_a$-th Fourier coefficient of 
$E(h_\infty,f_{l,H};s=l)$ is zero unless $a\in \Z$ 
since $E(h,f_{l,H};s=l)$ is of level one. 
\end{proof}

\begin{remark}For $h_\infty=
\left(
\begin{array}{cc}
1 & x \\
0 & 1 
\end{array}
\right)
\left(
\begin{array}{cc}
\sqrt{\sqrt{2}y} & 0 \\
0 & \sqrt{\sqrt{2}y}^{-1} 
\end{array}
\right) \in B_H(\R)$ with $x\in \R$ and $y\in \R_{>0}$ 
(where $\sqrt{2}y$ is due to Pollack's normalization of $u_\eta(t,m)$), it is easy to see that 
$$W_{2\pi n e'}(h_\infty)=e^{2\pi n i x} (\sqrt{2}y)^{\frac{l+1}{2}}\sum_{-l\le v\le l \atop v\in 2\Z}
i^v  K_v(2\sqrt{2}\pi n y) X^{l-v}Y^{l+v} .$$
The partial sum in Theorem \ref{partialSum} may be also expressed in the terminology
of \cite[Section 7.2]{Miyake} by directly computing the Jacquet integral 
$$\ds\int_{\R}\overline{\psi_n(t)}f_{l,H,\infty}(w' n_H(t)h_\infty)dt=
\psi_n(x)\ds\int_{\R}\overline{\psi_n(t)}f_{l,H,\infty}(w' n_H(t)\diag(\sqrt{\sqrt{2}y},
\sqrt{\sqrt{2}y}^{-1}))dt$$ with 
$$f_{l,H,\infty}(w' n_H(t)\diag(\sqrt{\sqrt{2}y},
\sqrt{\sqrt{2}y}^{-1}))=|1-i (\sqrt{2}y)^{-1}t|^{-(l+1)}\Big(\frac{1+i (\sqrt{2}y)^{-1}t}{1-i (\sqrt{2}y)^{-1}t}\Big)^{\frac{v}{2}}
\times i^v,
$$ 
and may be of independent interest.
\end{remark}

\subsection{Constant term of Eisenstein series}
\label{section constant term of Eisenstein series}

By Pollack \cite{Po}, the constant term $E_0$ of $E_{\ell}(g,s=\ell+1)$ is the sum of the following two parts:
\begin{itemize}
    \item $f_{l}(g,\Phi_\f,s=l+1)$,
    \item $E_l^M(g,\Phi_\f,s=l+1)$. It is an Eisenstein series on $\SO(2,n)$, which is a holomorphic modular form of weight $l$.
\end{itemize}

\section{Siegel series}\label{Siegel}

In this section we compute the Siegel series from Section \ref{p-adic}.

\subsection{The case when $n$ is even}
Now suppose $\dim\, V'=n+2=2m$. Since our orthogonal group splits at any $p$, we may assume that 
$$q(x_1,...,x_m,y_1,...,y_m)=x_1y_1+\cdots+x_my_m.$$ 
Let $\eta=(a_1,...,a_m,b_1,...,b_m)$.
Then (\ref{main}) is 
\begin{equation}\label{quad}
B_{r,\eta}=\sum_{u_1,...,u_m,v_1,...,v_m\in\Bbb Z/p^r\Bbb Z\atop u_1v_1+\cdots+u_mv_m\equiv 0 \, \text{mod $p^r$}}
e^{2\pi i \frac 1p(a_1v_1+\cdots+a_mv_m+b_1u_1+\cdots+b_mu_m)}.
\end{equation}
Now note that 
$$
\sum_{x\in\Bbb Z/p^r\Bbb Z}
e^{2\pi i \frac {x}{p^r}(u_1v_1+\cdots+u_mv_m)}=\begin{cases} p^r, &\text{if $u_1v_1+\cdots+u_mv_m\equiv 0$ (mod $p^r$)}\\ 0, &\text{otherwise.}\end{cases}
$$
Hence the above sum is

\begin{eqnarray} 
&& p^{-r} \sum_{u_1,...,u_m,v_1,...,v_m\in\Bbb Z/p^r\Bbb Z}
e^{2\pi i \frac 1{p^r}(a_1v_1+\cdots+a_mv_m+b_1u_1+\cdots+b_mu_m)}\sum_{x\in\Bbb Z/p^r\Bbb Z} e^{2\pi i \frac {x}{p^r}(u_1v_1+\cdots+u_mv_m)} \nonumber\\
&&\phantom{xxxxxx} = p^{-r}\sum_{x\in \Bbb Z/p^r\Bbb Z} \prod_{i=1}^m \left(\sum_{u_i,v_i\in\Bbb Z/p^r\Bbb Z} e^{2\pi i \frac {x}{p^r}(u_iv_i)} e^{2\pi i \frac 1{p^r}(a_iv_i+b_iu_i)}\right).\label{sum1}
\end{eqnarray}

Now we have

\begin{lemma} \label{Gauss-sum}
Let $j=v_p(x)$. Then
$$
\sum_{u,v\in\Bbb Z/p^r\Bbb Z} e^{2\pi i \frac 1{p^r}(xuv+av+bu)}=\begin{cases} p^{2r}, &\text{if $j\geq r$, $p^r|a$, $p^r|b$}\\ 
p^{r+j} e^{2\pi i \frac {-a'b'x_0^{-1}}{p^{r-j}}}, &\text{if $0\leq j<r$, $p^j|a$, $p^j|b$}\\
0, &\text{otherwise}\end{cases},
$$
where $x=p^j x_0$ with $p\nmid x_0$ and $x_0^{-1}$ is taken modulo $p^{r-j}$, and $a=p^j a'$, $b=p^j b'$ when $j<r$.
\end{lemma}

\begin{proof}
Let $f(x,a,b)$ be the exponential sum. 
Put $N=p^r$, and $\zeta_N=e^{2\pi i /N}$. Then 
$$f(x,a,b)=\sum_{\text{$u$ (mod $N$)}} \zeta_N^{bu} \sum_{\text{$v$ (mod $N$)}} \zeta_N^{(xu+a)v}=N \sum_{\text{$u$ (mod $N$)}\atop \text{$xu\equiv -a$ (mod $N$)}} \zeta_N^{bu}.
$$
Case A: $j\geq r$. Then if $p^r\nmid a$, it is clear that $f(x,a,b)=0$. Suppose $p^r|a$. Then it is easy to see 
$$f(x,a,b)=N\sum_{\text{$u$ (mod $N$)}} \zeta_N^{bu}=\begin{cases} N^2=p^{2r}, &\text{if $p^r|b$}\\ 0, &\text{otherwise}\end{cases}.
$$
In conclusion, we have, if $j\geq r$,
$$f(x,a,b)=\begin{cases} p^{2r}, &\text{if $p^r|a$ and $p^r|b$}\\ 0, &\text{otherwise}\end{cases}.
$$

Case B: $0\leq j<r$. The congruence $xu\equiv -a$ (mod $N$) is solvable if and only if $p^j|a$. Then the solutions are
$$\text{$u\equiv u_0$ (mod $p^{r-j}$), $u_0\equiv -a'x_0^{-1}$ (mod $p^{r-j}$)},
$$
and there are exactly $p^j$ of them: $u=u_0+t p^{r-j}$ with $t=0,1,...,p^j-1$. Hence
$$f(x,a,b)=N\sum_{t=0}^{p^j-1} \zeta_N^{b(u_0+tp^{r-j})}=N\zeta_N^{bu_0} \sum_{t=0}^{p^j-1} e^{2\pi i \frac {bt}{p^j}}.
$$
Therefore,
$$f(x,a,b)=\begin{cases} Np^{j}\zeta_N^{bu_0}=p^{r+j} e^{2\pi i \frac {b'u_0}{p^{r-j}}}, &\text{if $p^j|a$ and $p^j|b$}\\ 0, &\text{otherwise}\end{cases}.
$$
This proves the lemma.
\end{proof}

So if $j<r$, the sum over elements $x=p^j c$, $p\nmid c$ in (\ref{sum1}) is
$$
\begin{cases} p^{rm-r+jm}\sum_{c\in\Bbb Z/p^{r-j}\Bbb Z,\, p\nmid c } e^{-2\pi i \frac 1{p^{r+j}} c^{-1}(a_1b_1+\cdots+a_mb_m)}, &\text{if $a_i\equiv 0$\, mod $p^j$, $b_i\equiv 0$\, mod $p^j$ for each $i$}\\ 0, &\text{otherwise}.
\end{cases}
$$
By setting $a_i=p^j a_i', b_i=p^j b_i'$ for each $i$, the above sum is,
\begin{equation}\label{Ramanujan}
\sum_{c\in\Bbb Z/p^{r-j}\Bbb Z, \, p\nmid c} e^{2\pi i \frac 1{p^{r-j}} c(a_1'b_1'+\cdots+a_m'b_m')}.
\end{equation}
 It is the Ramanujan sum. By \cite[p. 44]{IK}, it is
$$\sum_{d|(p^{r-j},A_j)} d \mu\Big(\frac {p^{r-j}}d\Big),
$$
where $\mu(n)$ is the M\"obius function and $A_j=a_1'b_1'+\cdots+a_m'b_m'$.
 Then
$$(\ref{Ramanujan})=\begin{cases} p^{r-j}-p^{r-j-1}, &\text{if $p^{r-j}|A_j$}\\ -p^{r-j-1}, &\text{if $p^{r-j}\nmid A_j$ but $p^{r-j-1}|A_j$}\\ 0, &\text{if $p^{r-j-1}\nmid A_j$ ($r-j>1$)}\end{cases}.
$$

Therefore, if $\eta\equiv 0$ (mod $p^r$), (\ref{sum1}) is 
$$p^{2rm-r}+\sum_{j=0}^{r-1} p^{rm-r+jm} (p^{r-j}-p^{r-j-1})
=p^{2rm-r}+\sum_{j=0}^{r-1} (p^{rm+jm-j}-p^{rm+jm-j-1}).
$$

Hence, $J(s,\eta,\Phi_p)=0$ unless $\eta\in V'(\Bbb Z_p)$. Now let $p^{k}\|\eta$, and $p^{2k+k'}\|q(\eta)$ for some $k'$.
Then
$$J(s,\eta,\Phi_p)=\sum_{r=0}^\infty p^{-rs} B_{r,\eta},
$$
where $B_{0,\eta}=1$, and if $r\leq k$, since $\eta\equiv 0$ (mod $p^r$),
$$B_{r,\eta}=p^{2rm-r}+\sum_{j=0}^{r-1} (p^{rm+jm-j}-p^{rm+jm-j-1}).
$$

If $r>k$,
(\ref{quad}) is 
$$B_{r,\eta}=\sum_{j=0}^k p^{rm-r+jm}\sum_{c\in\Bbb Z/p^{r-j}\Bbb Z,\, p\nmid c } e^{2\pi i \frac 1{p^{r+j}} 2c(a_1b_1+\cdots+a_mb_m)}
=\sum_{j=0}^k p^{rm-r+jm} \left(\sum_{d|(p^{r-j},A_j)} d \mu\Big(\frac {p^{r-j}}d\Big)\right),
$$
where $A_j=p^{-2j}q(\eta)$.

If $r\geq 2k+k'+2$, $B_{r,\eta}=0$. And if $r=2k+k'+1$,
$$B_{2k+k'+1,\eta}=-p^{(2k+k'+1)m-1}.
$$

If $r=2k+k'$,
$$B_{2k+k',\eta}=\begin{cases} p^{k'm-k'}(p^{k'}-p^{k'-1}), &\text{if $k=0$}\\
p^{rm-r}(p^{2k+k'}-p^{2k+k'-1})+p^{rm-r+m}(-p^{2k+k'-2})
=p^{rm}-p^{rm-k'}-p^{rm+m-2}, &\text{if $k>0$.}\end{cases}
$$

If $k< r\leq k+k'$, then $r-j\leq k+k'-j\leq 2k+k'-2j$. Hence 
$$
\sum_{d|(p^{r-j},A_j)} d \mu\Big(\frac {p^{r-j}}d\Big)=p^{r-j}-p^{r-j-1}.
$$
Therefore,
$$B_{r,\eta}=\sum_{j=0}^k p^{rm-r+jm}(p^{r-j}-p^{r-j-1}).
$$

If $k+k'< r\leq 2k+k'-1$, then $r-j\leq 2k+k'-2j$ implies $j\leq 2k+k'-r$. So
\begin{eqnarray*}
&& B_{r,\eta}=\sum_{j=0}^{2k+k'-r} p^{rm-r+jm}(p^{r-j}-p^{r-j-1})+p^{rm-r+m(2k+k'-r+1)} (-p^{2r-2k-k'-2}) \\
&& =\sum_{j=0}^{2k+k'-r} p^{rm}(p^{j(m-1)}-p^{j(m-1)-1}) -p^{r+(2k+k')(m-1)+m-2}.
\end{eqnarray*}
This is valid also for $r=2k+k'$.

When $r=1$,
$$
B_{1,\eta}=\begin{cases} p^{2m-1}+p^m-p^{m-1}, &\text{if $\eta\equiv 0$ (mod $p$)}\\ -p^{m-1}, &\text{if $\eta\not\equiv 0$ (mod $p$) and $q(\eta)\not\equiv 0$ (mod $p$)}\\ p^m-p^{m-1}, &\text{if $\eta\not\equiv 0$ (mod $p$) and $q(\eta)\equiv 0$ (mod $p$).}\end{cases}
$$

If $\eta\equiv 0$ (mod $p$),
$B_{1,\eta}=C(n)=\#\{ u\in V'(\Bbb Z/p\Bbb Z) | \, q(u)\equiv 0\, \text{mod $p$}\}$.
This gives another proof of Lemma 4.4.2 and Lemma 4.4.3 in \cite{Po}.

\noindent{\it Proof of Theorem \ref{Siegel-even}.}
Consider 
$J(s,\eta,\Phi_p)(1-p^{m-1-s})^{-1}$. Namely,
$$\left(\sum_{r=0}^\infty B_{r,\eta}p^{-rs}\right) (1-p^{m-1-s})^{-1}
=\sum_{a=0}^\infty \left(\sum_{r=0}^a B_{r,\eta} p^{(m-1)(a-r)}\right) p^{-as}.
$$

We claim that if $a\geq 2k+k'+1$, $\sum_{r=0}^a B_{r,\eta} p^{(m-1)(a-r)}=0$.
Then this shows that
$$Q_{\eta,p}(p^{m-s})=\sum_{a=0}^{2k+k'} \left(\sum_{r=0}^a B_{r,\eta} p^{(m-1)(a-r)}\right) p^{-as}.
$$

First, consider the case $k=0$. Then if $a\geq k'+1$, 

$$ \sum_{r=0}^a B_{r,\eta} p^{-r(m-1)}=1+\sum_{r=1}^{k'-1} (p^r-p^{r-1})+p^{k'}(1-p^{-1})-p^{k'}=0.
$$

Now let $k>0$. If $a\geq 2k+k'+1$, 
\begin{eqnarray*}
&& \sum_{r=0}^a B_{r,\eta} p^{-r(m-1)}=\sum_{r=0}^k \left(p^{rm}+\sum_{j=0}^{r-1} (p^{j(m-1)+r}-p^{j(m-1)+r-1})\right) +
\sum_{r=k+1}^{k+k'} \sum_{j=0}^k (p^{j(m-1)+r}-p^{j(m-1)+r-1}) \\
&& +\sum_{r=k+k'+1}^{2k+k'} \left(\sum_{j=0}^{2k+k'-r} (p^{j(m-1)+r}-p^{j(m-1)+r-1}) -p^{-rm+2r+(2k+k')(m-1)+m-2}\right) -p^{2k+k'}.
\end{eqnarray*}

By Mathematica, we can show that it is zero. By Mathematica, we can show that $Q_{\eta,p}(p^{m-s})$ is a polynomial in $p^{m-s}$:
$$Q_{\eta,p}(p^{m-s})=\sum_{a=0}^{2k+k'} A_{a,\eta} (p^{m-s})^a,
$$
where 
\begin{equation}\label{A_a}
A_{a,\eta}=\begin{cases} 1+p^{m-1}+\cdots+(p^{m-1})^a, &\text{if $0\leq a\leq k$,}\\
1+p^{m-1}+\cdots+(p^{m-1})^k, &\text{if $k+1\leq a\leq k+k'$,}\\
1+p^{m-1}+\cdots+(p^{m-1})^{2k+k'-a}, &\text{if $k+k'+1\leq a\leq 2k+k'$.}
\end{cases}
\end{equation}

Let $X=p^{m-s}$, and we can write
$$Q_{\eta,p}(X)=\sum_{a=0}^{k-1} A_{a,\eta} (X^a+X^{2k+k'-a})+\frac 12 A_{k,\eta} \sum_{a=k}^{k+k'} (X^a+X^{2k+k'-a}).
$$
Then it satisfies the functional equation $X^{2k+k'}Q_{\eta,p}(X^{-1})=Q_{\eta,p}(X).$

\subsection{The case when $n$ is odd} 
If $n$ is odd, let ${\rm dim}\, V'=n+2=2m+1$. In this case, we need to modify the previous 
argument. 
Since the orthogonal group splits at $p$, we may assume that 
$$q(u_1,...,u_m,v_1,...,v_m,u_0)=u_1v_1+\cdots+u_mv_m+u_0^2.
$$
For $\eta=(a_1,...,a_m,b_1,...,b_m,a_0)$, $u=(u_1,...,u_m,v_1,...,v_m,u_0)$,
$$(\eta,u)=a_1v_1+\cdots+a_mv_m+b_1u_1+\cdots+b_mu_m+2a_0u_0.
$$

Then 
$$
B_{r,\eta}= p^{-r}\sum_{x\in \Bbb Z/p^r\Bbb Z} \left(\sum_{u_0\in\Bbb Z/p^r\Bbb Z} e^{\frac {2\pi i}{p^r}(xu_0^2+2a_0u_0)}\right) 
\prod_{i=1}^m \left(\sum_{u_i,v_i\in\Bbb Z/p^r\Bbb Z} e^{2\pi i \frac {x}{p^r}(u_iv_i)} e^{2\pi i \frac 1{p^r}(a_iv_i+b_iu_i)}\right).
$$

Now we need to divide into two cases.

\subsubsection{The case when $p>2$ {\rm(}and $n$ is odd{\rm)}}
Recall the following exponential sums \cite{Wh}, a generalization of the Ramanujan sum (\ref{Ramanujan}): Let $p$ be an odd prime. Let $\sum_{\text{$n$ mod $p^\alpha$}}'$ mean the sum
$\displaystyle\sum_{\text{$n$ mod $p^\alpha$}\atop p\nmid n}$. Then if $\alpha$ is odd, and we put $h_1=\frac h{p^{\alpha-1}}$ 
when $p^{\alpha-1}\|h$, we have

$${\sum}'_{\text{$n$ mod $p^\alpha$}} \Big(\frac h{p^\alpha}\Big) e^{\frac {2\pi i hn}{p^\alpha}}=\begin{cases} 0, &\text{if $p^\alpha|h$}\\ i^{\frac {(p-1)^2}4} \Big(\frac {h_1}p\Big) p^{\alpha-\frac 12}, &\text{if $p^\alpha\nmid h$ but $p^{\alpha-1}|h$}\\ 0, &\text{if $p^{\alpha-1}\nmid h$ ($\alpha>1$)}\end{cases}.
$$

Now for $x=p^j c$ with $p\nmid c$ and $j<r$. Then
$
\sum_{u_0\in \Bbb Z/p^{r}\Bbb Z} e^{\frac {2\pi i}{p^r}(xu_0^2+a_0u_0)}=0$ unless $a_0\equiv 0$ (mod $p^j$). 
Let $a_0=p^j a_0'$. Let $c^{-1}$ be taken modulo $p^{r-j}$. Then
\begin{eqnarray}\label{p>2}
\sum_{u_0\in \Bbb Z/p^{r}\Bbb Z} e^{\frac {2\pi i}{p^r}(xu_0^2+2a_0u_0)}
= e^{-2\pi i \frac 1{p^{r-j}} (a_0')^2 c^{-1}} p^{\frac {r+j}2} \begin{cases} 1, &\text{if $r-j$ even}\\ (\frac {c}p) \sqrt{(\frac {-1}p)}, &\text{if $r-j$ odd}\end{cases}.
\end{eqnarray}

In summary, we have the following: Recall
$$J(s,\eta,\Phi_p)=\sum_{r=0}^\infty p^{-rs} B_{r,\eta},\quad B_{0,\eta}=1.
$$
Suppose $p^k\|\eta$ and $p^{2k+k'}\|q(\eta)$. Let $q(\eta)_1=q(\eta)/p^{2k+k'}$. 
If $k\geq r$, then $\eta\equiv 0$ (mod $p^r$). 
Then
$$B_{r,\eta}=p^{2rm}+\sum_{j=0\atop \text{$r-j$ even}}^{r-1} p^{rm+jm-r} (p^{r-j}-p^{r-j-1})p^{\frac {r+j}2}.
$$

If $k<r$,

$$B_{r,\eta}=\sum_{j=0\atop \text{$r-j$ even}}^{k} p^{rm+jm+\frac {-r+j}2} C_{r,j}+\sum_{j=0\atop \text{$r-j$ odd}}^{k} p^{rm+jm+\frac {-r+j}2} D_{r,j},
$$
where $A_j=p^{-2j}q(\eta)$, and 
$$C_{r,j}=\begin{cases} p^{r-j}-p^{r-j-1}, &\text{if $p^{r-j}| A_j$}\\ -p^{r-j-1}, &\text{if $p^{r-j}\nmid A_j$ but $p^{r-j-1}| A_j$}\\ 0, &\text{if $p^{r-j-1}\nmid A_j$ ($r-j>1$)}\end{cases},
$$ 
$$D_{r,j}=\begin{cases} 0, &\text{if $p^{r-j}| A_j$}\\ \Big(\frac {q(\eta)_1}p\Big) p^{r-j-\frac 12}, &\text{if $p^{r-j}\nmid A_j$ but $p^{r-j-1}| A_j$}\\ 0, &\text{if $p^{r-j-1}\nmid A_j$ ($r-j>1$)}\end{cases}.
$$

\noindent{\it Proof of Theorem \ref{Siegel-odd}.} Suppose $2k+k'$ is odd. Consider $(1-p^{2(m-s)})^{-1}J(s,\eta,\Phi_p)$. Then by change of variables,
$$(1-p^{2(m-s)})^{-1}J(s,\eta,\Phi_p)=\sum_{b=0}^\infty \left(\sum_{r=0\atop \text{$b-r$ even}}^b p^{-mr} B_{r,\eta}\right) p^{(m-s)b}.
$$

Let $C(k,k',b)=\ds\sum_{r=0\atop \text{$b-r$ even}}^b p^{-mr} B_{r,\eta}$. 
We show that if $b\geq 2k+k'$, $C(k,k',b)=0$.

We write $C(k,k',b)=S_{\le k}+S_{\ge k+1}$ where 
\[
S_{\le k}:=\sum_{\substack{0\le r\le k\\ r\equiv b(2)}}p^{-mr}B_{r,\eta},\quad
S_{\ge k+1}:=\sum_{\substack{k+1\le r\le b\\ r\equiv b(2)}}p^{-mr}B_{r,\eta}.
\]

We first evaluate $S_{\le k}$. 
From the definition of $B_{r,\eta}$ for $r\le k$ we have
\[
S_{\le k}
=\sum_{\substack{0\le r\le k\\ r\equiv b(2)}}p^{rm}
+\sum_{\substack{0\le r\le k\\ r\equiv b(2)}}\;
\sum_{\substack{0\le j\le r-1\\ r-j\ \mathrm{even}}}
p^{jm-r}\Bigl(p^{\,r-j}-p^{\,r-j-1}\Bigr)p^{\frac{r+j}{2}}.
\]
Write $r=j+2t$ with $t\ge1$. Then necessarily $r\equiv j\ (\mathrm{mod}\ 2)$, hence $j\equiv b\ (\mathrm{mod}\ 2)$. 
As for the second term of $S_{\le k}$, 
\[
p^{jm-r}\,p^{\frac{r+j}{2}}\Bigl(p^{\,r-j}-p^{\,r-j-1}\Bigr)
= p^{jm+t}-p^{jm+t-1}.
\]
Therefore, for each fixed $j$ with $0\le j\le k$ and $j\equiv b(2)$,
\[
\sum_{t=1}^{\lfloor (k-j)/2\rfloor}\bigl(p^{jm+t}-p^{jm+t-1}\bigr)
= p^{\,jm+\lfloor (k-j)/2\rfloor}-p^{jm}.
\]
Using 
\(
\sum_{\substack{0\le r\le k\\ r\equiv b(2)}}p^{rm}
=\sum_{\substack{0\le j\le k\\ j\equiv b(2)}}p^{jm}
\),
we have 
\begin{equation}\label{eq:Sle-par}
S_{\le k}
=\sum_{\substack{0\le j\le k\\ j\equiv b(2)}} p^{\,jm+\lfloor (k-j)/2\rfloor}.
\end{equation}

Next we study $S_{\ge k+1}$. 
For $r\ge k+1$, by definition
\[
S_{\ge k+1}
=\sum_{\substack{k+1\le r\le b\\ r\equiv b(2)}}\;
\sum_{\substack{0\le j\le k\\ r-j\ \mathrm{even}}}
p^{jm+\frac{-r+j}{2}}\;C_{r,j}.
\]
We write 
\[
S_{\ge k+1}=U+V,\ \text{ where }
\begin{cases}
U:=\displaystyle\sum_{\substack{k+1\le r\le b\\ r\equiv b(2)}}
\ \sum_{\substack{0\le j\le k\\ r-j\ \mathrm{even}\\ r+j\le 2k+k'}} 
p^{jm+\frac{-r+j}{2}}\bigl(p^{\,r-j}-p^{\,r-j-1}\bigr),\\[12pt]
V:=\displaystyle\sum_{\substack{k+1\le r\le b\\ r\equiv b(2)}}
\ \sum_{\substack{0\le j\le k\\ r-j\ \mathrm{even}\\ r+j=2k+k'+1}} 
\Bigl(-\,p^{jm+\frac{-r+j}{2}}\,p^{\,r-j-1}\Bigr).
\end{cases}
\]
Write $r=j+2t$ again. The parity condition $r\equiv b(2)$ yields $j\equiv b(2)$.  
The lower bound $r\ge k+1$ gives 
\(
t\ge t_{\min}(j):=\bigl\lceil (k+1-j)/2\bigr\rceil
\).
The inequality $r+j\le 2k+k'$ gives
\(
t\le t_{\max}(j):=\bigl\lfloor(2k+k'-2j)/2\bigr\rfloor
\).
As for $U$, since $b\geq 2k+k'$,
\[
p^{jm+\frac{-r+j}{2}}\bigl(p^{\,r-j}-p^{\,r-j-1}\bigr)=p^{jm+t}-p^{jm+t-1}.
\]
Thus
\begin{equation}\label{eq:U-par}
U=\sum_{\substack{0\le j\le k\\ j\equiv b(2)}}
\ \sum_{t=t_{\min}(j)}^{t_{\max}(j)}\bigl(p^{jm+t}-p^{jm+t-1}\bigr)
=\sum_{\substack{0\le j\le k\\ j\equiv b(2)}}\Bigl(p^{jm+t_{\max}(j)}-p^{jm+t_{\min}(j)-1}\Bigr).
\end{equation}
As for $V$, the condition $r+j=2k+k'+1$ yields $t=t_{\max}(j)+1$. Since $b\geq 2k+k'$,
\begin{equation}\label{eq:V-par}
V=\sum_{\substack{0\le j\le k\\ j\equiv b(2)}}
\Bigl(-\,p^{jm+t_{\max}(j)}\Bigr).
\end{equation}
Combining \eqref{eq:U-par} and \eqref{eq:V-par} yields
\begin{equation}\label{eq:Sge-par}
S_{\ge k+1}
= -\sum_{\substack{0\le j\le k\\ j\equiv b(2)}} p^{\,jm+t_{\min}(j)-1}, 
\qquad t_{\min}(j)=\Bigl\lceil\frac{k+1-j}{2}\Bigr\rceil.
\end{equation}

For every integer $0\le j\le k$,
\[
\Bigl\lfloor\frac{k-j}{2}\Bigr\rfloor
=\Bigl\lceil\frac{k+1-j}{2}\Bigr\rceil-1.
\]
Using \eqref{eq:Sle-par} and \eqref{eq:Sge-par},
\[
S_{\le k}
=\sum_{\substack{0\le j\le k\\ j\equiv b(2)}} p^{\,jm+\lfloor (k-j)/2\rfloor}
=\sum_{\substack{0\le j\le k\\ j\equiv b(2)}} p^{\,jm+t_{\min}(j)-1}
=-\,S_{\ge k+1}.
\]
Hence $C(k,k',b)=S_{\le k}+S_{\ge k+1}=0$.

For the functional equation, we compute the coefficients $C(k,k',b)$ explicitly case by case. We compute one case, $k$ even.
Let $k=2a$ and $2k+k'=2l+1$. Then we prove
\begin{equation}\label{n-odd-C}
C(k,k',b)=\begin{cases} \ds\sum_{u=0\atop \text{$b-u$ even}}^b p^{mu} p^{\frac {b-u}2}, &\text{if $b<k$}\\
\ds\sum_{u=0\atop \text{$b-u$ even}}^{k} p^{mu}p^{\frac {b-u}2}, &\text{if $k\leq b\leq k+k'-1$}\\
\ds\sum_{u=0\atop \text{$b-u$ even}}^{2k+k'-1-b} p^{mu}p^{\frac {b-u}2}, &\text{if $k+k'\leq b\leq 2k+k'-1$}.
\end{cases}
\end{equation}
Suppose $b$ is odd. Let $b=2j+1$. Then if $b\leq k$,
\begin{eqnarray*}
&& \sum_{r=0\atop \text{$r$ odd}}^b B_{r,\eta}p^{-rm}=\sum_{i=0}^j B_{2i+1,\eta}p^{-m(2i+1)}\\
&&=\sum_{i=0}^j (p^{m(2i+1)}+\sum_{u=0}^{i-1} p^{m(2u+1)}(p^{i-u}-p^{i-u-1}))=\sum_{u=0}^j p^{m(2u+1)}p^{j-u}.
\end{eqnarray*}

If $k\leq b<k+k'$,  
$$ \sum_{r=0\atop \text{$r$ odd}}^b B_{r,\eta}p^{-rm}=\sum_{i=0}^{a-1} +\sum_{i=a}^j.
$$
Here by change of variables,
$$\sum_{i=0}^{a-1}= \sum_{u=0}^{a-1} p^{m(2u+1)}p^{a-1-u}.
$$
$$\sum_{i=a}^j= \sum_{i=a}^j\sum_{u=0\atop i+u\leq l-1}^{a-1} p^{m(2u+1)}(p^{i-u}-p^{i-u-1})=\sum_{u=0}^{a-1} p^{m(2u+1)}(p^{j-u}-p^{a-u-1}),
$$
since $i+u\leq j+a-1\leq l-1$. So $i+u=l$ cannot occur.
Hence if $k\leq b<k+k'$, 
$$\sum_{r=0\atop \text{$r$ odd}}^b B_{r,\eta}p^{-rm}=\sum_{u=0}^{a-1} p^{m(2u+1)}p^{j-u}.
$$

If $b\geq k+k'$, 
$$ \sum_{r=0\atop \text{$r$ odd}}^b B_{r,\eta}p^{-rm}=\sum_{i=0}^{a-1} +\sum_{i=a}^j,
$$
where 
$$\sum_{i=0}^{a-1}= \sum_{u=0}^{a-1} p^{m(2u+1)}p^{a-1-u},
$$
$$\sum_{i=a}^j= \sum_{i=a}^j\sum_{u=0\atop i+u\leq l-1}^{a-1} p^{m(2u+1)}(p^{i-u}-p^{i-u-1})+
\sum_{i=a}^j\sum_{u=0\atop i+u=l}^{a-1} p^{m(2u+1)}(-p^{i-u-1}).
$$
Now the first sum is
\begin{eqnarray*}
&& \sum_{u=0}^{a-1}\sum_{i=a}^{\text{min$(j,l-1-u)$}} p^{m(2u+1)}(p^{i-u}-p^{i-u-1}))\\
&&= \sum_{u=0}^{l-1-j} p^{m(2u+1)}p^{j-u}+\sum_{u=l-j}^{a-1} p^{m(2u+1)}p^{l-2u-1}-\sum_{u=0}^{a-1} p^{m(2u+1)}p^{a-u-1}.
\end{eqnarray*}

In the second sum, if $j<l-u$, $u<l-j$, and it is an empty sum. If $j\geq l-u$, $u\geq l-j$, and it becomes
$\sum_{u=l-j}^{a-1} p^{m(2u+1)}(-p^{l-2u-1})$. By summing up, we have, if $b\geq k+k'$, 
$$ \sum_{r=0\atop \text{$r$ odd}}^b B_{r,\eta}p^{-rm}=\sum_{u=0}^{l-1-j} p^{m(2u+1)}p^{j-u}.
$$

The case of $b$ even is similar.

\medskip

Suppose $2k+k'$ is even. Let $\ve=\chi_\eta(p)$, and consider $(1+\ve p^{m-s})^{-1}J(s,\eta,\Phi_p)$. Then by change of variables,
$$(1+\ve p^{m-s})^{-1}J(s,\eta,\Phi_p)=\sum_{b=0}^\infty (-\ve)^b p^{(m-s)b} \left(\sum_{r=0}^b (-\ve)^r p^{-mr} B_{r,\eta}\right).
$$
Let $C^\ve(k,k',b)=\sum_{r=0}^b (-\ve)^r p^{-mr} B_{r,\eta}$.
We prove that $C^\ve(k,k',b)=0$ if $b\geq 2k+k'+1$.

We write
$C^\ve(k,k',b)=S_{\le k}+S_{\ge k+1}$ where 
\[
S_{\le k}:=\sum_{0\le r\le k}(-\varepsilon)^r p^{-mr}B_{r,\eta},\quad
S_{\ge k+1}:=\sum_{k+1\le r\le b}(-\varepsilon)^r p^{-mr}B_{r,\eta}.
\]

We first evaluate $S_{\le k}$:
\[
S_{\le k}
=\sum_{0\le r\le k}(-\varepsilon)^r p^{rm}
+\sum_{0\le r\le k}\;
\sum_{\substack{0\le j\le r-1\\ r-j\ \mathrm{even}}}
(-\varepsilon)^r\,p^{jm-r}\Bigl(p^{\,r-j}-p^{\,r-j-1}\Bigr)p^{\frac{r+j}{2}}.
\]
Put $r=j+2t$ with $t\ge1$. Then $r\equiv j\ (\mathrm{mod}\ 2)$ and $(-\varepsilon)^r=(-\varepsilon)^j$. 
Since
\[
p^{jm-r}\,p^{\frac{r+j}{2}}\Bigl(p^{\,r-j}-p^{\,r-j-1}\Bigr)
= p^{jm+t}-p^{jm+t-1}.
\]
 for each fixed $j$ with $0\le j\le k$,
\[
\sum_{t=1}^{\lfloor (k-j)/2\rfloor}\bigl(p^{jm+t}-p^{jm+t-1}\bigr)
= p^{\,jm+\lfloor (k-j)/2\rfloor}-p^{jm}.
\]
Using $\sum_{0\le r\le k}(-\varepsilon)^r p^{rm}
=\sum_{0\le j\le k}(-\varepsilon)^j p^{jm}$,
we obtain
$$
S_{\le k}
=\sum_{0\le j\le k}(-\varepsilon)^j\,p^{\,jm+\lfloor (k-j)/2\rfloor}.
$$

Next we evaluate $S_{\ge k+1}$.
For $r\ge k+1$, 
$$S_{\ge k+1}=\sum_{k+1\le r\le b}\;
\sum_{\substack{0\le j\le k\\ r-j\ \mathrm{even}}}
(-\varepsilon)^r\,p^{jm+\frac{-r+j}{2}}\,C_{r,j}
+\sum_{k+1\le r\le b}\;
\sum_{\substack{0\le j\le k\\ r-j\ \mathrm{odd}}}
(-\varepsilon)^r\,p^{jm+\frac{-r+j}{2}}\,D_{r,j}.
$$

Write $r=j+2t$ in the first double sum and $r=j+2t+1$ in the second. 
For the first sum, the inequalities $r\ge k+1$ and $r+j\le 2k+k'$ are equivalent to
\[
t_{\min}(j)\le t\le t_{\max}(j),
\qquad
t_{\min}(j):=\Bigl\lceil\frac{k+1-j}{2}\Bigr\rceil,\quad
t_{\max}(j):=\Bigl\lfloor\frac{2k+k'-2j}{2}\Bigr\rfloor .
\]
On this range,
\[
(-\varepsilon)^r\,p^{jm+\frac{-r+j}{2}}\,C_{r,j}
=(-\varepsilon)^j\bigl(p^{jm+t}-p^{jm+t-1}\bigr).
\]
Hence for a fixed $j$,
$$
\sum_{t=t_{\min}(j)}^{t_{\max}(j)}(-\varepsilon)^r\,p^{jm+\frac{-r+j}{2}}\,C_{r,j}
=(-\varepsilon)^j\Bigl(p^{jm+t_{\max}(j)}-p^{jm+t_{\min}(j)-1}\Bigr).
$$

For the second sum (odd $r-j$), nonzero terms occur only when $r+j=2k+k'+1$ and it corresponds to $t=t_{\max}(j)$. Using $(-\varepsilon)^r\varepsilon=-(-\varepsilon)^j$, we obtain
$$
\sum_{\substack{r=j+2t+1\\ r+j=2k+k'+1}}(-\varepsilon)^r\,p^{jm+\frac{-r+j}{2}}\,D_{r,j}
= -\,(-\varepsilon)^j\,p^{\,jm+t_{\max}(j)}.
$$
Since $b>2k+k'$, summing over all $j$ yields
$$
S_{\ge k+1}
=-\sum_{0\le j\le k}(-\varepsilon)^j\,p^{\,jm+t_{\min}(j)-1}.
$$
For every integer $0\le j\le k$,
\[
\Bigl\lfloor\frac{k-j}{2}\Bigr\rfloor
=\Bigl\lceil\frac{k+1-j}{2}\Bigr\rceil-1
=t_{\min}(j)-1.
\]
Hence
\[
S_{\le k}
=\sum_{0\le j\le k}(-\varepsilon)^j\,p^{\,jm+\lfloor (k-j)/2\rfloor}
=\sum_{0\le j\le k}(-\varepsilon)^j\,p^{\,jm+t_{\min}(j)-1}
=-\,S_{\ge k+1}.
\]
Therefore, $C^\ve(k,k',b)=S_{\le k}+S_{\ge k+1}=0$.

For the functional equation, let $2k+k'=2l$, and we compute explicitly $C^\ve(k,k',b)$: For $b<k$,
\begin{equation}\label{n-odd-C-epsilon}
C^\ve(k,k',b)
=\sum_{0\le j\le b}(-\varepsilon)^j\,p^{\,jm+\lfloor (b-j)/2\rfloor}.
\end{equation}
For $k\leq b\leq k+k'$, $r+j\leq 2l$. Hence in this range, there is no contribution from $D_{r,j}=0$, and

\begin{eqnarray*}
&& S_{\ge k+1}=\sum_{k+1\le r\le b}\;
\sum_{\substack{0\le j\le k\\ r-j\ \mathrm{even}}}
(-\varepsilon)^r\,p^{jm}(p^{\frac{r-j}2}-p^{\frac {r-j}2-1})\\
&&\phantom{xxxxx} =\sum_{j=0}^k (-\varepsilon)^j p^{jm}(p^{\lfloor\frac{b-j}2\rfloor}-p^{\lceil\frac {k+1-j}2\rceil-1})
=-S_{\le k}+ \sum_{j=0}^k (-\varepsilon)^j p^{jm+\lfloor\frac{b-j}2\rfloor}.
\end{eqnarray*}
Hence for $k\leq b\leq k+k'$,
$C^\ve(k,k',b)=\sum_{j=0}^{k} (-\ve)^j p^{jm+\lfloor \frac {b-j}2\rfloor}.
$
For $0\leq i\leq k'$, since $k'$ is even, $\lfloor \frac {k+k'-i-j}2\rfloor=\lfloor \frac {k+i-j}2\rfloor+\frac {k'}2-i$. Hence 
$$C^\ve(k,k',k+k'-i)=p^{k+k'-i-l}C^\ve(k,k',k+i).
$$

For $k+k'<b\leq 2k+k'$, 
\begin{eqnarray*}
S_{\ge k+1}=\sum_{k+1\le r\le b}\;
\sum_{\substack{0\le j\le k\\ r-j\ \mathrm{even}}\atop \text{$r+j\leq 2l$}}
(-\varepsilon)^r\,p^{jm}(p^{\frac{r-j}2}-p^{\frac {r-j}2-1})
+\sum_{k+1\le r\le b}\;
\sum_{\substack{0\le j\le k\\ \text{$r+j=2l+1$}}}
(-\varepsilon)^r\,p^{jm}(\varepsilon p^{\frac {r-j-1}2}).
\end{eqnarray*}
The second sum is 
$-\sum_{j=2l-b+1}^{k} (-\ve)^j p^{jm+l-j}.$ The first sum is
\begin{eqnarray*}
&& \sum_{j=0}^{2l-b} (-\varepsilon)^j p^{jm} \sum_{r=k+1}^b (p^{\frac{r-j}2}-p^{\frac {r-j}2-1})
+\sum_{j=2l-b+1}^k (-\varepsilon)^j p^{jm} \sum_{r=k+1}^{2l-j} (p^{\frac{r-j}2}-p^{\frac {r-j}2-1})\\
&& \phantom{xxxxxxxxxxxxx} =-S_{\leq k}+\sum_{j=0}^{2l-b} (-\varepsilon)^j p^{\lfloor \frac {b-j}2\rfloor}+ \sum_{j=2l-b+1}^{k} (-\ve)^j p^{jm+l-j}.
\end{eqnarray*}
Therefore,
if $k+k'<b\leq 2k+k'$, 
$$C^\ve(k,k',b)=\sum_{j=0}^{2l-b} (-\ve)^j p^{jm+\lfloor \frac {b-j}2\rfloor}.
$$
Hence, for $0\leq i<k$, $C^\ve(k,k',2k+k'-i)=p^{l-i}C^\ve(k,k',i).$ This proves the functional equation.
$\square$

\subsubsection{The case when $p=2$ {\rm(}and $n$ is odd{\rm)}}
When $p=2$, we need the analogue of (\ref{p>2}), the following generalized Gauss sums.
\begin{lemma}\label{p=2}
For $r\ge1$, and $x,a\in\mathbb{Z}$, define
\[
G(x,a):=\sum_{u\ (\mathrm{mod}\ 2^r)} \exp\!\Bigl(\frac{2\pi i}{2^r}\bigl(xu^2+2au\bigr)\Bigr).
\]
When $x\neq 0$, write $x=2^j x_0$ with $2\nmid x_0$. 
\[
G(x,a)=
\begin{cases}
2^r, & j\ge r,\  2^{r-1}\mid a,\, r\ge 1, \\[2pt]
2^r, & j=r-1, \ 2^j\|2a, \\[2pt]
2^{\frac {r+j}2}
\exp\!\Bigl(-\frac{2\pi i}{2^{r-j}}\cdot (a')^{2}x_0^{-1}\Bigr)\,\Big(\frac{2}{x_0}\Big)^{r-j}(1+i^{\,x_0}), 
& j\leq r-2,\, 2^{j+1}|2a,\\
0, & \text{otherwise}
\end{cases}
\]
where $(\frac 2{x_0})$ is the Kronecker symbol, and $x_0^{-1}$ is taken modulo $2^{r-j}$, and $a=2^j a'$.
\end{lemma}
\begin{proof}
If $j\ge r$, then $x\equiv 0\ (\mathrm{mod}\ 2^r)$ and the claim is easy to check. 
Assume $0\le j<r$. Write $x=2^j x_0$ with $2\nmid x_0$, we write 
\[
u=u_0+2^{\,r-j}t,\qquad u_0\ (\mathrm{mod}\ 2^{\,r-j}),\quad t=0,1,\dots,2^{j}-1 .
\]
Then we see that $x u^2 + 2a u \equiv 2^j x_0 u_0^2 + 2a u_0 + 2a 2^{\,r-j} t \pmod{2^r}.$
Therefore
$$
G(x,a)
=\sum_{u_0\ (\mathrm{mod}\ 2^{\,r-j})}\exp\!\Bigl(\frac{2\pi i}{2^{\,r-j}}\bigl(x_0u_0^2+(2a/2^j)u_0\bigr)\Bigr)
\cdot \sum_{t=0}^{2^{j}-1}\exp\!\Bigl(\frac{2\pi i}{2^{j}}\,2a\,t\Bigr).
$$
The latter sum in $t$ equals $0$ unless $2^j\mid 2a$, and equals $2^j$ when $2a=2^j a''$. Thus, if $2^j\nmid 2a$ then $G(x,a)=0$. If $2a=2^j a''$, then
$$
G(x,a)=2^j\sum_{u_0\ (\mathrm{mod}\ 2^{\,r-j})}\exp\!\Bigl(\frac{2\pi i}{2^{\,r-j}}\bigl(x_0u_0^2+a'' u_0\bigr)\Bigr).
$$

Put $T:=r-j$ and $a'':=a/2^{\,j}$. Then
\[
G(x,a)=2^j\,S_T(x_0,a''),
\qquad
S_T(x_0,a''):=\sum_{u_0\ (\mathrm{mod}\ 2^{T})}\exp\!\Bigl(\frac{2\pi i}{2^{T}}\,\bigl(x_0u_0^2+a''u_0\bigr)\Bigr).
\]
If $T=1$, a direct check gives $S_1(x_0,a'')=2$ for $a''$ odd and $0$ for $a''$ even.  
If $T\ge2$ and $a''$ is odd, we pair $u_0$ with $u_0+2^{T-1}$ to see cancellation so that 
$S_T(x_0,a'')=0$.  
If $T\ge 2$ and $a''$ is even (say $a''=2a'$), we have 
$x_0u_0^2+a''u_0=x_0(u_0+a' x_0^{-1})^2-x_0(a' x_0^{-1})^2,$ where $x_0^{-1}$ is taken modulo $2^T$. Hence
\[
S_T(x_0,a'')
=\exp\!\Bigl(-\frac{2\pi i}{2^{T}}\,(a')^2 x_0^{-1}\Bigr)\sum_{v\ (\mathrm{mod}\ 2^{T})}\exp\!\Bigl(\frac{2\pi i}{2^{T}}\,x_0 v^2\Bigr)
=\exp\!\Bigl(-\frac{2\pi i}{2^{T}}\cdot (a')^{2}x_0^{-1}\Bigr)\,
\Big(\frac{2}{x_0}\Big)^T(1+i^{\,x_0})\,2^{T/2},
\]
where the final identity uses a well-known formula for quadratic Gauss sums 
(cf. \cite[Theorem 1.5.4, p.27]{BEW}).
\end{proof}

We also need the following \cite[p. 29]{BEW}: 
\begin{lemma}
Let $\chi$ be a nontrivial Dirichlet character mod $k$, and conductor $f$.
Let $\tau_k(m,\chi)=\sum_{n=1}^{k-1} \chi(n) e^{2\pi imn/k}$. Then
$$\tau_k(m,\chi)=\begin{cases} kf^{-1}\tau_f(fm/k,\chi), &\text{if $k|fm$}\\ 0, &\text{otherwise}\end{cases}.
$$
Moreover, $\tau_f(m,\chi)=\overline{\chi}(m)\tau_f(1,\chi)$. We have the following:
$\tau_4(1,\chi)=2i$ if $\chi$ is non-trivial. 
Let $\chi$ be a primitive character (mod 8) normalized so that $\chi(5)=-1$. Then
$\tau_8(1,\chi)=2\sqrt{2} i^{\frac {1-\chi(-1)}2}$.
\end{lemma}

Then

\begin{prop} Suppose $T\geq 4$ is even and $m>0$. Then
$$\sum_{x\in \Bbb Z/2^T\Bbb Z} \exp\left(-\frac {2\pi i}{2^T} xm\right)(\tfrac 2x)^T (1+i^x)=\begin{cases} 2^T-2^{T-1}, &\text{if $2^T|m$}\\
-2^{T-1}, &\text{if $2^{T-1}\|m$}\\
2^{T-1}(\frac {-1}{m2^{2-T}}), &\text{if $2^{T-2}\|m$}\\
0, &\text{otherwise}.
\end{cases}
$$
When $T=2$, we have
$$\sum_{x\in \Bbb Z/4\Bbb Z} \exp\left(-\frac {2\pi i}{4} xm\right)(\tfrac 2x)^2 (1+i^x)=\begin{cases} 2, &\text{if $4|m$}\\
-2, &\text{if $2\|m$}\\
2(\frac {-1}m), &\text{if $2\nmid m$}.
\end{cases}
$$

Suppose $T\geq 5$ is odd. Then
$$\sum_{x\in \Bbb Z/2^T\Bbb Z} \exp\left(-\frac {2\pi i}{2^T} xm\right)(\tfrac 2x)^T (1+i^x)
=\begin{cases} 2^{T-\frac 32} (\frac 2{m2^{3-T}}) (1+(\frac {-1}{m2^{3-T}})), &\text{if $2^{T-3}\|m$}\\
0, &\text{otherwise}.
\end{cases}
$$
When $T=3$, we have
$$\sum_{x\in \Bbb Z/8\Bbb Z} \exp\left(-\frac {2\pi i}{8} xm\right)(\tfrac 2x)^3 (1+i^x)
=\begin{cases} 2^{\frac 32} (\frac 2m) (1+(\frac {-1}m)), &\text{if $2\nmid m$}\\
0, &\text{otherwise}.
\end{cases}
$$
\end{prop}

\begin{proof} We only prove when $T$ is odd. 
Notice the sum runs over odd $x$ since $(\tfrac 2x)=0$ for $x$ even. 
Since $i^x=(\frac {-1}x)i$ for $x$ odd, write the sum as
$$\sum_{x\in \Bbb Z/2^T\Bbb Z} e^{-\frac {2\pi i}{2^T} xm}(\tfrac 2x) + i\sum_{x\in \Bbb Z/2^T\Bbb Z} e^{-\frac {2\pi i}{2^T} xm}(\tfrac {-2}x).
$$
Let $\chi(x)=(\tfrac 2x)$ and $\chi'(x)=(\tfrac {-2}x)$. They are primitive characters mod 8, and $\chi(5)=-1, \chi'(5)=-1$. 
Then apply the above lemma.
\end{proof}

\noindent{\it Proof of Theorem \ref{Siegel-2}.}
Suppose $2^k\|\eta$ and $2^{2k+k'}\|q(\eta)$. 
Recall the following:
If $r\leq k$, then
$$B_{r,\eta}=2^{2mr}+\sum_{j=0\atop \text{$r-j$ even}}^{r-1} 2^{mr+mj+\frac {-r+j}2}(2^{r-j}-2^{r-j-1}).
$$
If $r>k$, we need to divide into two cases:

Case 1. $2^{k+1}|a_1,...,a_m,b_1,...,b_m$, but $2^k\|a_0;$
In this case, by congruence, $k'=0$. 
If $r\geq k+3$,
$$B_{r,\eta}=\sum_{j=0\atop \text{$r-j$ even}}^{k} 2^{rm+jm+\frac {-r+j}2} C^{(2)}_{r,j}+\sum_{j=0\atop \text{$r-j$ odd}}^{k} 2^{rm+jm+\frac {-r+j}2} D^{(2)}_{r,j},
$$
where $A_j=2^{-2j}q(\eta)$, and [Note that since $j\leq k$, $r-j\geq 3$.]
$$C^{(2)}_{r,j}=\begin{cases} 2^{r-j}-2^{r-j-1}, &\text{if $2^{r-j}| A_j$}\\ -2^{r-j-1}, &\text{if $2^{r-j-1}\| A_j$}\\
2^{r-j-1}(\frac {-1}{A_j 2^{2-r+j}}), &\text{if $2^{r-j-2}\| A_j$}\\
 0, &\text{if $2^{r-j-2}\nmid A_j$}\end{cases},
$$ 
$$D^{(2)}_{r,j}=\begin{cases} 2^{r-j-\frac 32}(\tfrac 2{A_j2^{3-r+j}})(1+(\tfrac {-1}{A_j2^{3-r+j}})), &\text{if $2^{r-j-3}\| A_j$ and $r-j\geq 3$}\\
 0, &\text{otherwise.}\end{cases}
$$

We also have
$$B_{k+1,\eta}=2^{2m(k+1)}+\sum_{j=0\atop \text{$k+1-j$ even}}^{k} 2^{(k+1)m+jm+\frac {-k-1+j}2} (2^{k+1-j}-2^{k-j}).
$$
$$B_{k+2,\eta}=2^{(2k+3)m}+\sum_{j=0\atop \text{$k+2-j$ even}}^{k} 2^{(k+2)m+jm+\frac {-k-2+j}2} (2^{k+2-j}-2^{k+1+j}).
$$

Case 2. $2^k\|(a_1,...,a_m,b_1,...,b_m)$ and $2^k|a_0$.
If $r\geq k+2$,
$$B_{r,\eta}=\sum_{j=0\atop \text{$r-j$ even}}^{k} 2^{rm+jm+\frac {-r+j}2} C^{(2)}_{r,j}+\sum_{j=0\atop \text{$r-j$ odd}}^{k} 2^{rm+jm+\frac {-r+j}2} D^{(2)}_{r,j}.
$$

We have
$$B_{k+1,\eta}=\sum_{j=0\atop \text{$k+1-j$ even}}^{k} 2^{(k+1)m+jm+\frac {-k-1+j}2} (2^{k+1-j}-2^{k-j}).
$$

Now if $k'$ is odd, then $\chi_\eta(2)=0$. Consider 
$$(1-2^{2(m-s)})^{-1}J(s,\eta,\Phi_2)=\sum_{b=0}^\infty C(k,k',b)2^{(m-s)b},
$$
where $\displaystyle C(k,k',b)=\sum_{r=0\atop \text{$b-r$ even}}^b 2^{-mr} B_{r,\eta}$. 
Then as in odd prime case, we can show
\begin{equation}\label{n-odd-C}
C(k,k',b)=\begin{cases} \ds\sum_{u=0\atop \text{$b-u$ even}}^b 2^{mu} 2^{\frac {b-u}2}, &\text{if $b<k$}\\
\ds\sum_{u=0\atop \text{$b-u$ even}}^{k} 2^{mu}2^{\frac {b-u}2}, &\text{if $k\leq b\leq k+k'-1$}\\
\ds\sum_{u=0\atop \text{$b-u$ even}}^{2k+k'-1-b} 2^{mu}2^{\frac {b-u}2}, &\text{if $k+k'\leq b\leq 2k+k'-1$}.
\end{cases}
\end{equation}
So the functional equation of $Q_{\eta,2}(X)$ follows.

\medskip

If $k'$ is even, and $q(\eta)2^{-2k-k'}\equiv 3$ (mod 4), then $\chi_\eta(2)=0$, and it is similar to odd $k'$ case.

\medskip

When $k'$ is even, and $q(\eta)2^{-2k-k'}\equiv 1$ (mod 8), then $\chi_\eta(2)=1$. Then
$$Q_\eta(2^{m-s})=(1+2^{m-s})^{-1}J(s,\eta,\Phi_2)=\sum_{b=0}^\infty (-1)^b 2^{(m-s)b} C(k,k',b),
$$
where $C(k,k',b)=\sum_{r=0}^b (-1)^r 2^{-mr} B_{r,\eta}$. Let $2k+k'=2l$.

We divide into two cases:

Case 1. $2^{k+1}|a_1,...,a_m,b_1,...,b_m$, but $2^k\|a_0$. In this case, $k'=0$.

For $b<k$,
\begin{equation}\label{n-odd-C-epsilon}
C(k,0,b)
=\sum_{j=0}^b (-1)^j\,2^{\,jm+\lfloor (b-j)/2\rfloor}.
\end{equation}

If $b\geq k+3$, write
$$C(k,k',b)=S_{\leq k}+ (-1)^{k+1} 2^{-m(k+1)} B_{k+1,\eta}+
(-1)^{k+2} 2^{-m(k+2)} B_{k+2,\eta}+S_{\geq k+3, b},
$$
where 
\[
S_{\le k}:=\sum_{0\le r\le k} (-1)^r 2^{-mr}B_{r,\eta},\quad
S_{\ge k+3,b}:=\sum_{k+3\le r\le b} (-1)^r 2^{-mr}B_{r,\eta}.
\]

As in the odd prime case, we show 
$$
S_{\le k}
=\sum_{0\le j\le k} (-1)^j\,2^{\,jm+\lfloor (k-j)/2\rfloor}.
$$

We have
\begin{eqnarray*}
&& (-1)^{k+1} 2^{-m(k+1)} B_{k+1,\eta}+
(-1)^{k+2} 2^{-m(k+2)} B_{k+2,\eta} \\
&& =-(-1)^k \left(2^{m(k+1)}+\sum_{j=0\atop \text{$k+1-j$ even}}^k 2^{jm+\frac {k-1-j}2}\right)
+(-1)^k \left(2^{m(k+1)}+\sum_{j=0\atop \text{$k-j$ even}}^k 2^{jm+\frac {k-j}2}\right).
\end{eqnarray*}

Here 
$$-(-1)^k \left(\sum_{j=0\atop \text{$k+1-j$ even}}^k 2^{jm+\frac {k-1-j}2}
-\sum_{j=0\atop \text{$k-j$ even}}^k 2^{jm+\frac {k-j}2}\right)=S_{\leq k}.
$$

Hence $C(k,0,k+2)=2 C(k,0,k)$.

For $b\geq k+3$,

\begin{eqnarray*}
&& S_{\ge k+3,b}=\sum_{k+3\le r\le b}\;
\sum_{\substack{0\le j\le k\\ r-j\ \mathrm{even}}\atop \text{$r+j\leq 2k$}}
(-1)^r\,2^{jm}(2^{\frac{r-j}2}-2^{\frac {r-j}2-1})
+\sum_{k+3\le r\le b}\;
\sum_{\substack{0\le j\le k\\ r+j=2k+2}}
(-1)^r\,2^{jm+\frac {r-j}2-1} \\
&&\phantom{xxxsssss}+ \sum_{k+3\le r\le b}\;
\sum_{\substack{0\le j\le k\\ r+j=2k+3}}
(-1)^r\,2^{jm+\frac{r-j-1}2}.
\end{eqnarray*}
Now the sum of the first two sums is
\begin{eqnarray*}
&& \sum_{j=0}^{2k+2-b} (-1)^j 2^{jm}\sum_{r=k+3\atop \text{$r-j$ even}}^b 2^{\frac {r-j}2-1}
+\sum_{j=2k+3-b}^{k-1} (-1)^j 2^{jm}\sum_{r=k+3\atop \text{$r-j$ even}}^{2k+2-j} 2^{\frac {r-j}2-1} \\
&& \phantom{xxxx} =\sum_{j=0}^{2k+2-b} (-1)^j 2^{jm} (2^{\lfloor \frac {b-j}2\rfloor}-2^{\lceil \frac {k+1-j}2\rceil})
+\sum_{j=2k+3-b}^{k-1} (-1)^j 2^{jm} (2^{k-j+1}-2^{\lceil \frac {k+1-j}2\rceil}).
\end{eqnarray*}

The third sum is
$$-\sum_{j=2k+3-b}^k (-1)^j 2^{jm+k-j+1}.
$$

Therefore,
$$
S_{\ge k+3,b}=-2S_{\leq k}+\sum_{j=0}^{2k+2-b} (-1)^j 2^{jm+\lfloor \frac {b-j}2\rfloor}.
$$
Here we use the fact that for every integer $0\le j\le k$,
$
\Bigl\lfloor\frac{k-j}{2}\Bigr\rfloor
=\Bigl\lceil\frac{k+1-j}{2}\Bigr\rceil-1.
$
Hence 
$$C(k,0,b)=\sum_{j=0}^{2k+2-b} (-1)^j 2^{jm+\lfloor \frac {b-j}2\rfloor}.
$$
If $b>2k+2$, $S_{\ge k+3,b}=-2S_{\leq k}$, and hence $C(k,0,b)=0$.

Case 2. $2^k\|(a_1,...,a_m,b_1,...,b_m)$ and $2^k|a_0$.

As in the odd prime case, for $b\leq k$, 
$$
C(k,k',b)
=\sum_{j=0}^b (-1)^j\,2^{\,jm+\lfloor (b-j)/2\rfloor}.
$$

Let $C(k,k',k)=S_{\leq k}$.
For $b=k+1$,
\begin{eqnarray*}
 C(k,k',k+1)=S_{\leq k}+(-1)^{k+1} 2^{-m(k+1)} B_{k+1,\eta} 
 =S_{\leq k}+(-1)^{k+1} \sum_{j=0\atop \text{$k+1-j$ even}}^k 2^{jm+\frac {k-j-1}2}.
\end{eqnarray*}

Now 
$$(-1)^{k+1}2^{-m(k+1)}B_{k+1,\eta}+(-1)^{k+2}2^{-m(k+2)}B_{k+2,\eta}=\sum_{j=0}^k (-1)^j\,2^{\,jm+\lfloor (k-j)/2\rfloor}=S_{\leq k},
$$
so that $C(k,k',k+2)=2S_{\leq k}$. 

For $b\geq k+3$, write 
$C(k,k',b)=C(k,k',k+2)+S_{\geq k+3,b}$, where
\begin{eqnarray*}
&& S_{\geq k+3,b}=\sum_{r=k+3}^b (-1)^r 2^{-mr} B_{r,\eta} \\
&& =\sum_{k+3\le r\le b}\;
\sum_{\substack{0\le j\le k\\ r-j\ \mathrm{even}}\atop \text{$r+j\leq 2l$}}
(-1)^r\,2^{jm}(2^{\frac{r-j}2}-2^{\frac {r-j}2-1})
+\sum_{k+3\le r\le b}\;
\sum_{\substack{0\le j\le k\\ r+j=2l+2}}
(-1)^r\,2^{jm+\frac {r-j}2-1} \\
&&\phantom{xxx}+ \sum_{k+3\le r\le b}\;
\sum_{\substack{0\le j\le k\\ r+j=2l+3}}
(-1)^r\,2^{jm+\frac{r-j-1}2}.
\end{eqnarray*}

Suppose $k+2\leq b\leq k+k'$. Then $r+j\leq 2k+k'=2l$. Hence
$$
S_{\geq k+3,b}=\sum_{k+3\le r\le b}\;
\sum_{\substack{0\le j\le k\\ r-j\ \mathrm{even}}\atop \text{$r+j\leq 2l$}}
(-1)^r\,2^{jm+\frac {r-j}2-1}
=\sum_{j=0}^k (-1)^j 2^{jm+\lfloor \frac {b-j}2\rfloor}-2S_{\leq k}.
$$
Hence for $k+2\leq b\leq k+k'$, 
$$C(k,k',b)=\sum_{j=0}^k (-1)^j 2^{jm+\lfloor \frac {b-j}2\rfloor}.
$$
We can write, for $k+2\leq i\leq k+k'$, 
$$C(k,k',2k+k'+2-i)=2^{l-k+1-i}C(k,k',i).
$$

For $b=k+k'+1$, 
$$C(k,k',k+k'+1)=2^{\frac {k'}2}C(k,k',k+1).
$$

For $k+k'+2\leq b\leq 2k+k'+2$, one can show, $0\leq i\leq k$,
$$C(k,k',2k+k'+2-i)=2^{l+1-i}C(k,k',i).
$$

\medskip

Similarly for $k'$ even and $q(\eta)2^{-2k-k'}\equiv 5$ (mod 8), we compute the coefficients of $Q_\eta(2^{m-s})=(1-2^{m-s})^{-1}J(s,\eta,\Phi_2)$ since $\chi_\eta(2)=-1$. Since it is similar to the above, we omit the details.


\begin{thebibliography}{99}


\bibitem{AG}H.~Atobe and W-T.~Gan, {\em Local theta correspondence of tempered representations and Langlands parameters}, Inv. Math. {\bf 210} (2017), no. 2, 341--415. 


\bibitem{BEW} B-C.~Berndt, R-J.~Evans, and K-S.~Williams, Gauss and Jacobi sums. Canadian Mathematical Society Series of Monographs and Advanced Texts. John Wiley \& Sons, Inc., New York, 1998. xii+583 pp.


\bibitem{BJ}A.~Borel and H.~Jacquet, 
{\em Automorphic forms and automorphic representations}, Proc. Sympos. Pure Math., XXXIII, Part 1, pp. 189--207, Amer. Math. Soc., Providence, RI, 1979.


\bibitem{CZ}R.~Chen and J.~Zou, {\em Arthur's multiplicity formula for even orthogonal and unitary groups}, J. Eur. Math. Soc. {\bf 27} (2025), no. 12, 4769--4843.



\bibitem{G} B. Gross, {\em Groups over $\Bbb Z$}, Inv. math. {\bf 124} (1996), 263--279.

\bibitem{Ik01}T.~Ikeda, {\em On the lifting of elliptic cusp forms to Siegel cusp forms of degree $2n$}, Ann. of Math. (2) {\bf 154} (2001), no. 3, 641--681.

\bibitem{Ik08} \bysame, {\em On the lifting of hermitian modular forms}, Comp. Math. {\bf 144} (2008), 1107--1154. 

\bibitem{Ik17} \bysame, {\em On the functional equation of the Siegel series}, J. Num. Th. {\bf 172} (2017), 44--62.

\bibitem{IK} H. Iwaniec and E. Kowalski, Analytic Number Theory, American Mathematical Society, Colloquium Publications, Vol {\bf 53}, 2004.

\bibitem{Katsurada}H.~Katsurada, An explicit formula for Siegel series. Amer. J. Math. {\bf 121} (1999), no. 2, 415--452. 

\bibitem{KY} H-H.~Kim and T.~Yamauchi, {\em Cusp forms on the exceptional group of type $E_7$}, Comp. Math. {\bf 152} (2016), no. 2, 223--254.


\bibitem{KY1} \bysame, {\em Higher level cusp forms on exceptional group of type $E_7$}, Kyoto J. Math. {\bf 63} (2023), no. 3, 579--614.

\bibitem{KY3} \bysame, {\em On explicit Fourier expansions of theta lifts to $\SO(3,n+1)$ arising from elliptic newforms of level one}, in preparation.





\bibitem{Miyake}T.~Miyake, Modular forms. Translated from the 1976 Japanese original by Yoshitaka Maeda. Reprint of the first 1989 English edition. Springer Monographs in Mathematics. Springer-Verlag, Berlin, 2006. x+335 pp.

\bibitem{MS} T.~Miyazaki and Y.~Saito, {\em Theta lifts to certain cohomological representations of indefinite orthogonal groups}, Res. Number Theory {\bf 10} (2024), no. 2, Paper No. 25, 27 pp.



\bibitem{Po}A.~Pollack, {\em Modular forms on indefinite orthogonal groups of rank three. With appendix ``Next to minimal representation" by G. Savin}, J. Num. Th. {\bf 238} (2022), 611--675. 

\bibitem{Po1}\bysame, {\em  The minimal modular form on quaternionic $E_8$},
J. Inst. Math. Jussieu {\bf 21} (2022), no. 2, 603-636.


\bibitem{Serre} J-P.~ Serre, A Course in Arithmetic. Graduate Texts in Mathematics, No. 7. Springer-Verlag, New York-Heidelberg, 1973. viii+115 pp. 


\bibitem{Wh} A.L. Whiteman, {\em A note on Kloosterman sums}, Bull. Amer. Math. Soc. {\bf 51} (1945), 373--377.

\bibitem{Y10}S.~Yamana,  {\em On the lifting of elliptic cusp forms to cusp forms on quaternionic unitary groups.} J. Num. Th. {\bf 130} (2010), no. 11, 2480--2527.

\end{thebibliography}
\end{document}